\documentclass[10pt]{article}

\usepackage{amsthm} 
\usepackage{amsmath} 
\usepackage{amssymb}
\usepackage{latexsym}

\allowdisplaybreaks

\newtheorem{theorem}{Theorem}[section]
\newtheorem{proposition}[theorem]{Proposition}
\newtheorem{corollary}[theorem]{Corollary}

\newtheorem{remark}[theorem]{Remark}
\newtheorem{lemma}[theorem]{Lemma}

\usepackage{xspace}

\usepackage[margin=1.0in]{geometry}

\title{A mixed problem for the Laplace operator in a domain with moderately close holes}

\author{ 
Matteo Dalla Riva \thanks{Department of Mathematics, The University of Tulsa, USA \& Department of Mathematics, Aberystwyth University, Ceredigion SY23 3BZ, Wales, UK.} \ and  Paolo Musolino\thanks{Dipartimento di Matematica, Universit\`a degli Studi di Padova,  Italy.}}

\date{\ }

\begin{document}

\maketitle

\noindent
{\bf Abstract:}  We investigate the behavior of the solution of a mixed problem in a domain with two moderately close holes. We introduce a positive  parameter $\epsilon$ and we define a perforated domain $\Omega_{\epsilon}$ obtained by making two small perforations in an open set. Both the size and the distance of the cavities tend to $0$ as $\epsilon \to 0$. For  $\epsilon$ small, we denote by $u_{\epsilon}$ the  solution of a mixed problem for the Laplace equation in $\Omega_{\epsilon}$. We describe what happens to $u_{\epsilon}$ as $\epsilon \to 0$  in terms of real analytic maps and we compute an asymptotic expansion. 
\vspace{9pt}

 \par
\vspace{9pt}


\section{Introduction}
\label{introd}
The analysis of  singular domain perturbation problems for linear equations and system of partial differential equations has caught the attention of several authors. In particular, a wide literature has been dedicated to the study of boundary value problems defined in domains with small holes or inclusions shrinking to points. This type of problems is of interest not only for the mathematical aspects but also in view of concrete applications to the investigation of physical models in fluid dynamics, in elasticity, and in thermodynamics. For example, problems on domains with small holes or inclusions can arise in the modeling of dilute composites or of perforated elastic bodies. In this paper, we will focus on a mixed problem for the Laplace operator in a bounded domain with two moderately close small holes. In other words, we will consider a domain with two cavities such that both their size and the distance between them tend to zero. However, we will assume that the perforations are `moderately close', \textit{i.e.}, the distance tends to zero `not faster' than the size.

In order to introduce the problem, we first define the geometric setting. We fix once for all  a natural number
\[
n\in {\mathbb{N}}\setminus\{0,1 \}\, .
\]
Then we consider $\alpha\in]0,1[$ and three subsets $\Omega^i_1$, $\Omega^i_2$,  $\Omega^o$ of ${\mathbb{R}}^{n}$ satisfying the following assumption:
\begin{equation}\label{dom}
\begin{split}
&\text{$\Omega^i_1$, $\Omega^i_2$,  $\Omega^o$ are bounded open connected subsets of ${\mathbb{R}}^{n}$}\\
&\text{of class $C^{1,\alpha}$ such that ${\mathbb{R}}^{n}\setminus{\mathrm{cl}}\Omega^i_1$, ${\mathbb{R}}^{n}\setminus{\mathrm{cl}}\Omega^i_2$ and ${\mathbb{R}}^{n}\setminus{\mathrm{cl}}\Omega^o$ are}\\
&\text{connected  and  that $0\in \Omega^i_1\cap\Omega^i_2 \cap \Omega^o$}.
\end{split}
\end{equation}
The letter `$i$' stands for `inner' and the letter `$o$' stands for `outer'. The symbol `${\mathrm{cl}}$' denotes the closure.  The set $\Omega^o$ will play the role of the `unperturbed' domain, where we make two perforations of the shape of $\Omega^i_1$ and of $\Omega^i_2$, respectively.  We also fix two points
\begin{equation}\label{p}
p^1, p^2 \in \Omega^o\, ,\qquad p^1 \neq p^2\, .
\end{equation}
Then we take $\epsilon_0>0$ and a function $\eta$ from $]0,\epsilon_0[$ to $]0,+\infty[$ such that
\begin{equation}\label{eta}
\lim_{\epsilon\to 0^+}\eta(\epsilon)=0 \qquad \text{and} \qquad \lim_{\epsilon \to 0^+} \frac{\epsilon}{\eta(\epsilon)}=r_\ast \in [0,+\infty[\, .
\end{equation}
The function $\eta$ will control the distance between the holes, while the parameter $\epsilon$ will determine their size. We assume that
\begin{equation}\label{assrast}
\left(p^1+r_\ast\mathrm{cl}\Omega^i_1\right)\cap\left(p^2+r_\ast\mathrm{cl}\Omega^i_2\right)= \emptyset\, .
\end{equation}
Possibly shrinking $\epsilon_0$, we may also assume that
\begin{equation}\label{e0}
\begin{split}
&\left(p^1+\frac{\epsilon}{\eta(\epsilon)}\mathrm{cl}\Omega^i_1\right)\cap\left(p^2+\frac{\epsilon}{\eta(\epsilon)}\mathrm{cl}\Omega^i_2\right)= \emptyset  \qquad \forall \epsilon \in ]0,\epsilon_0[\, ,\\
&\bigg (\eta(\epsilon)p^1+\epsilon\mathrm{cl}\Omega^i_1\bigg)\cup\bigg(\eta(\epsilon)p^2+\epsilon\mathrm{cl}\Omega^i_2\bigg) \subseteq \Omega^o \qquad \forall \epsilon \in ]0,\epsilon_0[\, .
\end{split}
\end{equation}
Then we introduce the perforated domain
\[
\Omega_{\epsilon}\equiv \Omega^o \setminus  \bigcup_{j=1}^2 \bigg (\eta(\epsilon)p^j+\epsilon\mathrm{cl}\Omega^i_j\bigg) \qquad \forall \epsilon\in ]0,\epsilon_{0}[\, .
\]
In other words, the set $\Omega_{\epsilon}$ is obtained by removing from $\Omega^o$ the two sets $\eta(\epsilon)p^1+\epsilon\mathrm{cl}\Omega^i_1$ and $\eta(\epsilon)p^2+\epsilon\mathrm{cl}\Omega^i_2$.  As $\epsilon \to 0^+$, both the size of the perforations and their distance tend to $0$. Next, for each $\epsilon$ positive and small enough, we want to introduce a mixed problem for the Laplace operator in $\Omega_\epsilon$. Namely, we consider a Dirichlet condition on $\partial \Omega^o$ and Neumann conditions on the boundary of the holes. Thus, we take a function $f_1 \in C^{0,\alpha}(\partial \Omega^i_1)$, a function $f_2  \in C^{0,\alpha}(\partial \Omega^i_2)$, a function $g$ in $C^{1,\alpha}(\partial \Omega^o)$, and for each $\epsilon \in ]0,\epsilon_0[$ we consider the following mixed problem:
\begin{equation}
\label{bvpe}
\left\{
\begin{array}{ll}
\Delta u(x)=0 & \forall x \in \Omega_{\epsilon}\,,\\
\frac{\partial }{\partial \nu_{\eta(\epsilon)p^j+\epsilon\Omega^i_j}}u(x)=f_j\Big(\big(x-\eta(\epsilon)p^j\big)/\epsilon\Big) & \forall x \in  \eta(\epsilon)p^j+\epsilon\partial\Omega^i_j\, ,  \forall j\in \{1,2\}\, ,\\
u(x)=g(x) & \forall x \in \partial \Omega^o\, ,
\end{array}
\right.
\end{equation}
where $\nu_{\eta(\epsilon)p^j+\epsilon\Omega^i_j}$ denotes the outward unit normal to $ \eta(\epsilon)p^j+\epsilon\partial\Omega^i_j$ for $j\in \{1,2\}$.

Then, if $\epsilon \in ]0,\epsilon_0[$,  problem \eqref{bvpe} has a unique solution in $C^{1,\alpha}(\mathrm{cl}\Omega_{\epsilon})$ and we denote such a solution by $u_{\epsilon}$.  We are interested in studying the behavior of $u_\epsilon$ as $\epsilon \to 0$ and thus we pose the following questions.
\begin{enumerate}
\item[(i)] Let $x$ be a fixed point in 
$\Omega^o\setminus\{0\}$. What  can be said of the map 
$\epsilon\mapsto u_{\epsilon}(x)$ when $\epsilon$ is close to $0$ and 
positive?
\item[(ii)]
Let $t$ be a fixed point in 
${\mathbb{R}}^{n}\setminus 
\cup_{j=1}^2(p^j +r_\ast \Omega^i_j)$. What  can be said of the map 
$\epsilon\mapsto u_{\epsilon}(\eta(\epsilon) t)$ when $\epsilon$ is close to $0$ and 
positive?
\item[(iii)]
Let $j \in \{1,2\}$. Let $t$ be a fixed point of $\mathbb{R}^n\setminus\Omega^i_j$ such that $p^j +r_\ast t \not \in (p^l+r_\ast \mathrm{cl}\Omega_l)$ if $l \neq j$. What can be said of the map $\epsilon\mapsto u_{\epsilon}(\eta(\epsilon)p^j+\epsilon t)$ when $\epsilon$ is close to $0$ and positive?
\end{enumerate}	
In a sense, question (i) concerns the `macroscopic' behavior of $u_{\epsilon}$ far from the holes $\eta(\epsilon)p^1+\epsilon \Omega^i_1$ and $\eta(\epsilon)p^2+\epsilon \Omega^i_2$, whereas question (ii) concerns the `microscopic' behavior of $u_{\epsilon}$ in proximity of centers of the perforations, and question (iii) concerns the `microscopic' behavior of $u_{\epsilon}$ in proximity of the boundary of one of the perforations.	

Boundary value problems in domains with small holes are typical in the frame of asymptotic analysis and are usually investigated by means of asymptotic expansion methods.    As an example, we mention the method of matching outer and inner asymptotic expansions proposed by Il'in (see \cite{Il78}, \cite{Il92}, and \cite{Il99}) and the compound asymptotic expansion method of Maz'ya, Nazarov, and Plamenevskij, which allows the treatment of general Douglis--Nirenberg elliptic boundary value problems in domains with  perforations and corners (cf.~\cite{MaNaPl00}). Moreover, in Kozlov, Maz'ya, and Movchan \cite{KoMaMo99} one can find the study of boundary value problems in domains depending on a small parameter $\epsilon$ in such a way that the limit regions as $\epsilon$ tends to $0$ consist of subsets of different space dimensions. More recently, Maz'ya, Movchan, and Nieves provided the  asymptotic analysis of Green's kernels in domains with small cavities by applying the method of mesoscale asymptotic approximations  (cf.~\cite{MaMoNi13}). We also mention  Bonnaillie-No\"el, Lacave, and Masmoudi \cite{BoLaMa}, Chesnel and Claeys \cite{ChCl14}, and Dauge, Tordeux, and Vial \cite{DaToVi10}.

Problems in perforated domains find several applications in the frame of shape and topological optimization. For a detailed analysis, we refer to Novotny and Soko\l owsky \cite{NoSo13}, where the authors analyze the topological derivative to study problems in elasticity and heat diffusion. The topological derivative is indeed defined as the first term of the asymptotic expansion of a given shape functional with respect to a parameter which measures the singular domain perturbation (as, \textit{e.g.}, the diameter of a hole). Moreover, for several applications to inverse problems we refer, \textit{e.g.}, to the monograph Ammari and Kang \cite{AmKa07}.

In particular, boundary value problems in domains with moderately close holes have been deeply studied in Bonnaillie-No\"el, Dambrine, Tordeux, and Vial \cite{BoDaToVi07, BoDaToVi09}, Bonnaillie-No\"el and Dambrine \cite{BoDa13}, and Bonnaillie-No\"el, Dambrine, and Lacave \cite{BoDaLa}, where the authors exploit the method of multiscale asymptotic expansions. More precisely, in \cite{BoDaToVi09} they carefully analyze the case when $\eta(\epsilon)=\epsilon^\beta$ for $\beta \in ]0,1[$ and they provide asymptotic expansions.

Here, instead, we answer the questions in (i), (ii), (iii) by representing  the maps of (i), (ii), (iii) in terms of real 
analytic maps and in terms of  known 
functions of $\epsilon$ (such as $\eta(\epsilon)$, $\epsilon/\eta(\epsilon)$, $\log \eta(\epsilon)$, {\it etc}.).
We observe that our approach does have its advantages. Indeed, if for 
example we know that the function in (i) equals for $\epsilon>0$ a real analytic function defined in a whole neighborhood of $\epsilon=0$, then we know that such a map can be expanded in  power series for $\epsilon$ small. Moreover, we emphasize that we do not make any assumption on the form of the function $\eta(\epsilon)$ and that, by setting $\varrho_1=\eta(\epsilon)$ and $\varrho_2=\epsilon/\eta(\epsilon)$, we can treat $\varrho_1$ and $\varrho_2$ as two independents variables and prove real analyticity results for the solution upon the pair $(\varrho_1,\varrho_2)$.  In particular, one can deduce asymptotic expansions in the new variable $(\varrho_1,\varrho_2)$ around $(0,r_\ast)$.

Such an approach has been carried out for problems for the Laplace operator in a domain with a small hole  (cf., \textit{e.g.}, \cite{DaMu12,DaMu15}, Lanza de Cristoforis \cite{La07,La10}), and has later been extended to problems related to the system of equations of the linearized elasticity (cf., \textit{e.g.}, the first-named author and Lanza de Cristoforis \cite{DaLa10a}) and to the Stokes system (cf., \textit{e.g.}, \cite{Da13}). Moreover, analyticity results have been obtained in the frame of perturbation problems in spectral theory (cf., \textit{e.g.}, Buoso and Provenzano \cite{BuPr} and Lamberti and Lanza de Cristoforis \cite{LaLa04}).

The paper is organized as follows. In Section \ref{nota}, we introduce some notation and in Section \ref{for} we introduce a more general formulation of our problem. In Section \ref{prel}, we introduce some preliminary results. In Section \ref{finteq}, we formulate our problem in terms of integral equations. In Section \ref{fure}, we prove our main result, which answers our questions (i), (ii), (iii) above, and in Section \ref{asy} we compute an asymptotic expansion of the solution for $n=2$ and $r_\ast=0$.

\section{Notation}\label{nota}

We  denote the norm on 
a   normed space ${\mathcal X}$ by $\|\cdot\|_{{\mathcal X}}$. Let 
${\mathcal X}$ and ${\mathcal Y}$ be normed spaces. We endow the  
space ${\mathcal X}\times {\mathcal Y}$ with the norm defined by 
$\|(x,y)\|_{{\mathcal X}\times {\mathcal Y}}\equiv \|x\|_{{\mathcal X}}+
\|y\|_{{\mathcal Y}}$ for all $(x,y)\in  {\mathcal X}\times {\mathcal 
Y}$, while we use the Euclidean norm for ${\mathbb{R}}^{n}$. The symbol ${\mathbb{N}}$ denotes the 
set of natural numbers including $0$. If $(i,j) \in \mathbb{N}^2$, we denote by $\delta_{i,j}$ the Kronecker symbol, defined by setting $\delta_{i,j}=1$ if $i=j$ and $\delta_{i,j}=0$ if $i \neq j$. Let 
${\mathbb{D}}\subseteq {\mathbb {R}}^{n}$. Then $\mathrm{cl}{\mathbb{D}}$ 
denotes the 
closure of ${\mathbb{D}}$, $\partial{\mathbb{D}}$ denotes the boundary of ${\mathbb{D}}$, and $\nu_{\mathbb{D}}$ denotes the outer unit normal to $\partial \mathbb{D}$, where it is defined. We also set ${\mathbb{D}}^{-}\equiv {\mathbb {R}}^{n}\setminus{\mathrm{cl}}{\mathbb{D}}$. For all $R>0$, $ x\in{\mathbb{R}}^{n}$, 
$x_{j}$ denotes the $j$-th coordinate of $x$, 
$| x|$ denotes the Euclidean modulus of $ x$ in
${\mathbb{R}}^{n}$, and ${\mathbb{B}}_{n}( x,R)$ denotes the ball $\{
y\in{\mathbb{R}}^{n}:\, | x- y|<R\}$. 
Let $\Omega$ be an open 
subset of ${\mathbb{R}}^{n}$. The space of $m$ times continuously 
differentiable real-valued functions on $\Omega$ is denoted by 
$C^{m}(\Omega,{\mathbb{R}})$, or more simply by $C^{m}(\Omega)$. 
Let $r\in {\mathbb{N}}\setminus\{0\}$.
Let $f\in \left(C^{m}(\Omega)\right)^{r}$. The 
$s$-th component of $f$ is denoted $f_{s}$, and $Df$ denotes the jacobian matrix
$\left(\frac{\partial f_s}{\partial
x_l}\right)_{  (s,l)\in \{1,\dots ,r\}\times \{1,\dots ,n\}       }$. 
For a multi-index  $\eta\equiv
(\eta_{1},\dots ,\eta_{n})\in{\mathbb{N}}^{n}$ we also set $|\eta |\equiv
\eta_{1}+\dots +\eta_{n}  $. Then $D^{\eta} f$ denotes
$\frac{\partial^{|\eta|}f}{\partial
x_{1}^{\eta_{1}}\dots\partial x_{n}^{\eta_{n}}}$.    The
subspace of $C^{m}(\Omega )$ of those functions $f$ whose derivatives $D^{\eta }f$ of
order $|\eta |\leq m$ can be extended with continuity to 
$\mathrm{cl}\Omega$  is  denoted $C^{m}(
\mathrm{cl}\Omega )$. 
The
subspace of $C^{m}(\mathrm{cl}\Omega ) $  whose
functions have $m$-th order derivatives that are
uniformly H\"{o}lder continuous  with exponent  $\alpha\in
]0,1]$ is denoted $C^{m,\alpha} (\mathrm{cl}\Omega )$  
(cf., \textit{e.g.},~Gilbarg and 
Trudinger~\cite{GiTr83}). The subspace of $C^{m}(\mathrm{cl}\Omega ) $ of those functions $f$ such that $f_{|{\mathrm{cl}}(\Omega\cap{\mathbb{B}}_{n}(0,R))}\in
C^{m,\alpha}({\mathrm{cl}}(\Omega\cap{\mathbb{B}}_{n}(0,R)))$ for all $R\in]0,+\infty[$ is denoted $C^{m,\alpha}_{{\mathrm{loc}}}(\mathrm{cl}\Omega ) $.  \par
Now let $\Omega $ be a bounded
open subset of  ${\mathbb{R}}^{n}$. Then $C^{m}(\mathrm{cl}\Omega )$ 
and $C^{m,\alpha }({\mathrm{cl}}
\Omega )$ are endowed with their usual norm and are well-known to be 
Banach spaces. 
We say that a bounded open subset $\Omega$ of ${\mathbb{R}}^{n}$ is of class 
$C^{m}$ or of class $C^{m,\alpha}$, if ${\mathrm{cl}}
\Omega$ is a 
manifold with boundary imbedded in 
${\mathbb{R}}^{n}$ of class $C^{m}$ or $C^{m,\alpha}$, respectively
 (cf., \textit{e.g.}, Gilbarg and Trudinger~\cite[\S 6.2]{GiTr83}). 
We denote by 
$
\nu_{\Omega}
$
the outward unit normal to $\partial\Omega$.  For standard properties of functions 
in Schauder spaces, we refer the reader to Gilbarg and 
Trudinger~\cite{GiTr83}.

\par
If $M$ is a manifold  imbedded in 
${\mathbb{R}}^{n}$ of class $C^{m,\alpha}$, with $m\geq 1$, 
$\alpha\in ]0,1[$, one can define the Schauder spaces also on $M$ by 
exploiting the local parametrizations. In particular, one can consider 
the spaces $C^{k,\alpha}(\partial\Omega)$ on $\partial\Omega$ for 
$0\leq k\leq m$ with $\Omega$ a bounded open set of class $C^{m,\alpha}$,
and the trace operator from $C^{k,\alpha}({\mathrm{cl}}\Omega)$ to
$C^{k,\alpha}(\partial\Omega)$ is linear and continuous. 
  We denote by $d\sigma$ the area element of a manifold imbedded in ${\mathbb{R}}^{n}$. Also, we find convenient to set
\[
C^{k,\alpha}(\partial \Omega)_{0}\equiv
\left\{
f\in C^{k,\alpha}(\partial \Omega):\,\int_{\partial\Omega}f\,d\sigma=0 
\right\}\,.
\]
For the 
definition and properties of real analytic maps, we refer to Deimling \cite[p.~150]{De85}. In particular, we mention that the 
pointwise product in Schauder spaces is bilinear and continuous, and 
thus real analytic (cf., \textit{e.g.}, Lanza de Cristoforis and Rossi \cite[pp.~141, 142]{LaRo04}).\par

Let
$S_{n}$ be the function from ${\mathbb{R}}^{n}\setminus\{0\}$ to ${\mathbb{R}}$ defined by
\[
S_{n}(x)\equiv
\left\{
\begin{array}{lll}
\frac{1}{s_{n}}\log |x| \qquad &   \forall x\in 
{\mathbb{R}}^{n}\setminus\{0\},\quad & {\mathrm{if}}\ n=2\,,
\\
\frac{1}{(2-n)s_{n}}|x|^{2-n}\qquad &   \forall x\in 
{\mathbb{R}}^{n}\setminus\{0\},\quad & {\mathrm{if}}\ n>2\,,
\end{array}
\right.
\]
where $s_{n}$ denotes the $(n-1)$-dimensional measure of 
$\partial{\mathbb{B}}_{n}(0,1)$. $S_{n}$ is well-known to be a  
fundamental solution of the Laplace operator. \par

We now introduce the simple layer potential. Let  $\alpha\in]0,1[$. Let $\Omega$ be a bounded open subset of ${\mathbb{R}}^{n}$ of class $C^{1,\alpha}$. If $\mu\in C^{0}(\partial\Omega)$, we set
\[
v[\partial\Omega,\mu](x)\equiv
\int_{\partial\Omega}S_{n}(x-y)\mu(y)\,d\sigma_{y}
\qquad\forall x\in {\mathbb{R}}^{n}\,.
\]
As is well-known, if $\mu\in C^{0}(\partial{\Omega})$, then $v[\partial\Omega,\mu]$ is continuous in  ${\mathbb{R}}^{n}$. Moreover, if $\mu\in C^{0,\alpha}(\partial\Omega)$, then the function 
$v^{+}[\partial\Omega,\mu]\equiv v[\partial\Omega,\mu]_{|{\mathrm{cl}}\Omega}$ belongs to $C^{1,\alpha}({\mathrm{cl}}\Omega)$, and the function 
$v^{-}[\partial\Omega,\mu]\equiv v[\partial\Omega,\mu]_{|\mathbb{R}^n \setminus \Omega}$ belongs to $C^{1,\alpha}_{\mathrm{loc}}
(\mathbb{R}^n \setminus \Omega)$. 

\section{A more general formulation}\label{for}

In this section, we formulate a more general version of the problem we are interested in. Then, by the analysis of such a new problem, we are able to deduce our results concerning the behavior of the solution $u_\epsilon$ for $\epsilon$ close to $0$. In a sense, what we are going to do it is to replace $\eta(\epsilon)$ by $\varrho_1$ and $\epsilon/\eta(\epsilon)$ by $\varrho_2$, and to analyze the dependence of the solution of the problem upon $\varrho_1$ and $\varrho_2$, which we think as two independent variables.

Let $\alpha\in]0,1[$. Let $\Omega^i_1$, $\Omega^i_2$, $\Omega^o$ be as in \eqref{dom}. Let $p^1$, $p^2$ be as in \eqref{p}. Let $r_\ast \in [0,+\infty[$ be such that assumption \eqref{assrast} holds. Then we fix an open neighborhood $\tilde{\mathcal{U}}$ of $(0,r_\ast)$ in $\mathbb{R}^2$, such that
\begin{equation}\label{tildeU}
\begin{split}
&\left(p^1+\varrho_2\mathrm{cl}\Omega^i_1\right)\cap\left(p^2+\varrho_2\mathrm{cl}\Omega^i_2\right)= \emptyset  \qquad \forall (\varrho_1,\varrho_2) \in \tilde{\mathcal{U}}\, ,\\
&\bigg (\varrho_1 p^1+\varrho_1 \varrho_2 \mathrm{cl}\Omega^i_1\bigg)\cup\bigg(\varrho_1 p^2+\varrho_1 \varrho_2\mathrm{cl}\Omega^i_2\bigg) \subseteq \Omega^o \qquad \forall (\varrho_1,\varrho_2)\in \tilde{\mathcal{U}}\, .
\end{split}
\end{equation}
Then we introduce the perforated domain
\[
\Omega(\varrho_1,\varrho_2)\equiv \Omega^o \setminus  \bigcup_{j=1}^2 \bigg (\varrho_1p^j+\varrho_1 \varrho_2\mathrm{cl}\Omega^i_j\bigg) \qquad \forall (\varrho_1,\varrho_2) \in \tilde{\mathcal{U}}\, .
\]
Next we take a function $f_1 \in C^{0,\alpha}(\partial \Omega^i_1)$, a function $f_2  \in C^{0,\alpha}(\partial \Omega^i_2)$, a function $g$ in $C^{1,\alpha}(\partial \Omega^o)$, and for each pair $(\varrho_1,\varrho_2) \in \tilde{\mathcal{U}} \cap ]0,+\infty[^2$ we consider the following mixed problem
\begin{equation}
\label{bvprho}
\left\{
\begin{array}{ll}
\Delta u(x)=0 & \forall x \in \Omega(\varrho_1,\varrho_2)\,,\\
\frac{\partial }{\partial \nu_{\varrho_1p^j+\varrho_1 \varrho_2\Omega^i_j}}u(x)=f_j\Big(\big(x-\varrho_1p^j\big)/(\varrho_1 \varrho_2)\Big) & \forall x \in   \varrho_1p^j+\varrho_1 \varrho_2\partial\Omega^i_j,  \forall j\in \{1,2\} ,\\
u(x)=g(x) & \forall x \in \partial \Omega^o\, ,
\end{array}
\right.
\end{equation}
where $\nu_{\varrho_1p^j+\varrho_1 \varrho_2\Omega^i_j}$ denotes the outward unit normal to $\varrho_1p^j+\varrho_1 \varrho_2\partial\Omega^i_j$ for $j\in \{1,2\}$. If $(\varrho_1, \varrho_2) \in \tilde{\mathcal{U}} \cap ]0,+\infty[^2$,  problem \eqref{bvprho} has a unique solution in $C^{1,\alpha}(\mathrm{cl}\Omega(\varrho_1,\varrho_2))$ and we denote such a solution by $u[\varrho_1,\varrho_2]$. Clearly, if $\eta$, $r_\ast$ are as in \eqref{eta} and if $\epsilon_{0}$ is such that $(\eta(\epsilon), \epsilon/ \eta(\epsilon)) \in \tilde{\mathcal{U}} \cap ]0,+\infty[^2$ for all $\epsilon \in ]0,\epsilon_0[$, then
\[
\Omega_{\epsilon}=\Omega(\eta(\epsilon), \epsilon/ \eta(\epsilon)) \qquad \text{and} \qquad u_{\epsilon}=u[\eta(\epsilon), \epsilon/ \eta(\epsilon)] \, ,
\]
for all $\epsilon \in ]0,\epsilon_0[$.

\section{Preliminaries}\label{prel}

In this section we collect some preliminary results concerning mixed problems for the Laplace operator.

First of all, by the Divergence Theorem, we deduce the following uniqueness result.

\begin{proposition}\label{prop:uniq}
Let $\alpha \in ]0,1[$. Let $\mathcal{O}^i$, $\mathcal{O}^o$ be bounded open subsets of $\mathbb{R}^n$ of class $C^{1,\alpha}$ such that $\mathcal{O}^o$, $\mathbb{R}^n \setminus \mathrm{cl}\mathcal{O}^i$, and $\mathbb{R}^n \setminus \mathrm{cl}\mathcal{O}^o$ are connected and that $\mathrm{cl}\mathcal{O}^i \subseteq \mathcal{O}^o$. Let $v \in C^{1,\alpha}(\mathrm{cl}\mathcal{O}^o \setminus \mathcal{O}^i)$ be such that
\begin{equation}
\label{eq:uniq}
\left\{
\begin{array}{ll}
\Delta v(x)=0 & \forall x \in \mathcal{O}^o\setminus \mathrm{cl}\mathcal{O}^i\,,\\
\frac{\partial }{\partial \nu_{\mathcal{O}^i}}v(x)=0 & \forall x \in  \partial  \mathcal{O}^i\, ,\\
v(x)=0 & \forall x \in \partial \mathcal{O}^o\, .
\end{array}
\right.
\end{equation}
Then $v=0$ in $\mathrm{cl}\mathcal{O}^o \setminus \mathcal{O}^i$.
\end{proposition}

In the following lemma, we collect some well-known results of classical potential theory (cf.~Folland \cite[Ch.~3]{Fo95}, Lanza de Cristoforis and Rossi \cite[Thm.~3.1]{LaRo04}, Miranda \cite[Thm~5.I]{Mi65}).

\begin{lemma}\label{lem:smp}  
Let $\alpha \in ]0,1[$. Let $\Omega$ be a bounded open  subset of $\mathbb{R}^n$ of class $C^{1,\alpha}$.  Then the following statements hold.
\begin{enumerate}
\item[(i)]  The map from $C^{0,\alpha}(\partial \Omega)$ to $C^{1,\alpha}(\mathrm{cl}\Omega)$ which takes $\mu$ to $v^+[\partial \Omega, \mu]$ is linear and continuous. Similarly, if $\tilde{\Omega}$ is a bounded open  subset of $\mathbb{R}^n\setminus\mathrm{cl}\Omega$, then the map from $C^{0,\alpha}(\partial \Omega)$ to $C^{1,\alpha}(\mathrm{cl}\tilde\Omega)$ which takes $\mu$ to $v^-[\partial \Omega, \mu]_{|\mathrm{cl}\tilde\Omega}$ is linear and continuous. 
\item[(ii)] Let $\Omega$ be connected. The map from $C^{0,\alpha}(\partial\Omega)_0\times \mathbb{R}$ to $C^{1,\alpha}(\partial\Omega)$ which takes $(\mu,\xi)$ to $v[\partial \Omega, \mu]_{|\partial\Omega}+\xi$ is a linear homeomorphism. 
\item[(iii)] Let $\mathbb{R}^n \setminus \mathrm{cl}\Omega$ be connected. Then the map from $C^{0,\alpha}(\partial \Omega)$ to $C^{0,\alpha}(\partial \Omega)$ which takes $\mu$ to the function 
\[
\frac{1}{2}\mu(x)+\int_{\partial \Omega}DS_n(x-y)\nu_{\Omega}(x)\mu(y)\, d\sigma_y
\]
of the variable $x \in \partial \Omega$, is a linear homeomorphism.
\end{enumerate}
\end{lemma}

We now introduce and study an integral operator which we use in order to solve a mixed problem by means of simple layer potentials.

\begin{proposition}\label{prop:J}
Let $\alpha \in ]0,1[$.  Let $\mathcal{O}^i$, $\mathcal{O}^o$ be bounded open subsets of $\mathbb{R}^n$ of class $C^{1,\alpha}$ such that $\mathcal{O}^o$, $\mathbb{R}^n \setminus \mathrm{cl}\mathcal{O}^i$, and $\mathbb{R}^n \setminus \mathrm{cl}\mathcal{O}^o$ are connected and that $\mathrm{cl}\mathcal{O}^i \subseteq \mathcal{O}^o$. Let $J\equiv(J_1,J_2)$ be the operator from $C^{0,\alpha}(\partial \mathcal{O}^i)\times C^{0,\alpha}(\partial \mathcal{O}^o)_0\times \mathbb{R}$ to  $C^{0,\alpha}(\partial \mathcal{O}^i)\times C^{1,\alpha}(\partial \mathcal{O}^o)$ defined by
\[
\begin{split}
J_1[\mu_1,\mu_2,\xi](x)\equiv& \frac{1}{2}\mu_1(x)+\int_{\partial \mathcal{O}^i}DS_n(x-y)\nu_{\mathcal{O}^i}(x)\mu_1(y)\, d\sigma_y\\ &+\int_{\partial \mathcal{O}^o}DS_n(x-y)\nu_{\mathcal{O}^i}(x)\mu_2(y)\, d\sigma_y \qquad \forall x \in \partial \mathcal{O}^i\, ,\\
J_2[\mu_1,\mu_2,\xi](x)\equiv &\int_{\partial \mathcal{O}^i}S_n(x-y)\mu_1(y)\, d\sigma_y+\int_{\partial \mathcal{O}^o}S_n(x-y)\mu_2(y)\, d\sigma_y\\
&+\xi \qquad \forall x \in \partial \mathcal{O}^o\, ,
\end{split}
\]
for all $(\mu_1,\mu_2,\xi)\in C^{0,\alpha}(\partial \mathcal{O}^i)\times C^{0,\alpha}(\partial \mathcal{O}^o)_0\times \mathbb{R}$. Then $J$ is a linear homeomorphism.
\end{proposition}
\begin{proof} We first prove that $J$ is a Fredholm operator of index $0$. Let $\hat{J}\equiv(\hat{J}_1,\hat{J}_2)$ be the operator from $C^{0,\alpha}(\partial \mathcal{O}^i)\times C^{0,\alpha}(\partial \mathcal{O}^o)_0\times \mathbb{R}$ to  $C^{0,\alpha}(\partial \mathcal{O}^i)\times C^{1,\alpha}(\partial \mathcal{O}^o)$ defined by
\[
\begin{split}
\hat{J}_1[\mu_1,\mu_2,\xi](x)\equiv& \frac{1}{2}\mu_1(x)+\int_{\partial \mathcal{O}^i}DS_n(x-y)\nu_{\mathcal{O}^i}(x)\mu_1(y)\, d\sigma_y \quad \forall x \in \partial \mathcal{O}^i\, ,\\
\hat{J}_2[\mu_1,\mu_2,\xi](x)\equiv &\int_{\partial \mathcal{O}^o}S_n(x-y)\mu_2(y)\, d\sigma_y+\xi \qquad \forall x \in \partial \mathcal{O}^o\, ,
\end{split}
\]
for all $(\mu_1,\mu_2,\xi)\in C^{0,\alpha}(\partial \mathcal{O}^i)\times C^{0,\alpha}(\partial \mathcal{O}^o)_0\times \mathbb{R}$. By Lemma \ref{lem:smp} (ii), (iii) one can show that $\hat{J}$ is a linear homeomorphism. Then let $\tilde{J}\equiv(\tilde{J}_1,\tilde{J}_2)$ be the operator from $C^{0,\alpha}(\partial \mathcal{O}^i)\times C^{0,\alpha}(\partial \mathcal{O}^o)_0\times \mathbb{R}$ to  $C^{0,\alpha}(\partial \mathcal{O}^i)\times C^{1,\alpha}(\partial \mathcal{O}^o)$ defined by
\[
\begin{split}
\tilde{J}_1[\mu_1,\mu_2,\xi](x)\equiv& \int_{\partial \mathcal{O}^o}DS_n(x-y)\nu_{\mathcal{O}^i}(x)\mu_2(y)\, d\sigma_y \quad \forall x \in \partial \mathcal{O}^i\, ,\\
\tilde{J}_2[\mu_1,\mu_2,\xi](x)\equiv &\int_{\partial \mathcal{O}^i}S_n(x-y)\mu_1(y)\, d\sigma_y \qquad \forall x \in \partial \mathcal{O}^o\, ,
\end{split}
\]
for all $(\mu_1,\mu_2,\xi)\in C^{0,\alpha}(\partial \mathcal{O}^i)\times C^{0,\alpha}(\partial \mathcal{O}^o)_0\times \mathbb{R}$. By classical potential theory and standard calculus in Schauder spaces, one can show that $\tilde{J}$ is a compact operator. Since $J=\hat{J}+\tilde{J}$,  we deduce that $J$ is a Fredholm operator of index $0$. As a consequence, in order to prove that $J$ is a linear homeormorphism, it suffices to show that it is injective. So let $(\mu_1,\mu_2,\xi)\in C^{0,\alpha}(\partial \mathcal{O}^i)\times C^{0,\alpha}(\partial \mathcal{O}^o)_0\times \mathbb{R}$ be such that $J[\mu_1,\mu_2,\xi]=(0,0)$. Then by classical potential theory, the function $v\equiv v[\partial \mathcal{O}^i,\mu_1]_{|\mathrm{cl}\mathcal{O}^o \setminus \mathcal{O}^i}+v[\partial \mathcal{O}^o,\mu_2]_{|\mathrm{cl}\mathcal{O}^o \setminus \mathcal{O}^i}+ \xi$ is a solution in $C^{1,\alpha}(\mathrm{cl}\mathcal{O}^o \setminus \mathcal{O}^i)$ of problem \eqref{eq:uniq}. Accordingly, $v[\partial \mathcal{O}^i,\mu_1]_{|\mathrm{cl}\mathcal{O}^o \setminus \mathcal{O}^i}+v[\partial \mathcal{O}^o,\mu_2]_{|\mathrm{cl}\mathcal{O}^o \setminus \mathcal{O}^i}+ \xi=0$ in $\mathrm{cl}\mathcal{O}^o \setminus \mathcal{O}^i$, and so
\begin{equation}\label{v-=-v+-xi}
v^-[\partial \mathcal{O}^i,\mu_1]=-v^+[\partial \mathcal{O}^o,\mu_2]- \xi\qquad\textrm{in $\mathrm{cl}\mathcal{O}^o \setminus \mathcal{O}^i$}\, .
\end{equation}
Also, $v[\partial \mathcal{O}^i,\mu_1]=-v[\partial \mathcal{O}^o,\mu_2]- \xi$ on $\partial \mathcal{O}^i$ and by uniqueness of the solution of the Dirichlet problem for the Laplace operator, we deduce
\begin{equation}\label{v+=-v+-xi}
v^+[\partial \mathcal{O}^i,\mu_1]=-v^+[\partial \mathcal{O}^o,\mu_2]- \xi\qquad\textrm{in $\mathrm{cl}\mathcal{O}^i$}\, .
\end{equation}
As a consequence, $v[\partial \mathcal{O}^i,\mu_1]=-v^+[\partial \mathcal{O}^o,\mu_2]- \xi$ on the whole of $\mathrm{cl} \mathcal{O}^o$. Since $v^+[\partial \mathcal{O}^o,\mu_2]$ is in $C^{1,\alpha}(\mathrm{cl} \mathcal{O}^o)$ (cf.~Lemma \ref{lem:smp}), we have
\begin{equation}\label{-dv++dv+=0}
-\frac{\partial}{\partial \nu_{\mathcal{O}^i}}v^+[\partial \mathcal{O}^o,\mu_2]_{|\mathrm{cl}\mathcal{O}^o \setminus \mathcal{O}^i}+\frac{\partial}{\partial \nu_{\mathcal{O}^i}}v^+[\partial \mathcal{O}^o,\mu_2]_{|\mathrm{cl}\mathcal{O}^i}=0\qquad\textrm{on $\partial\mathcal{O}^i$}\,.
\end{equation}
By equalities \eqref{v-=-v+-xi} and \eqref{v+=-v+-xi}, and by standard jump properties of the single layer potential, the expression on the left hand side of \eqref{-dv++dv+=0} equals
\begin{equation}\label{dv--dv+=mu}
\frac{\partial}{\partial \nu_{\mathcal{O}^i}}v^-[\partial \mathcal{O}^i,\mu_1]-\frac{\partial}{\partial \nu_{\mathcal{O}^i}}v^+[\partial \mathcal{O}^i,\mu_1]=\mu_1\qquad\textrm{on $\partial\mathcal{O}^i$}\, .
\end{equation} Hence, by \eqref{-dv++dv+=0} and \eqref{dv--dv+=mu} it follows that $\mu_1=0$. Thus $v[\partial \mathcal{O}^o,\mu_2]+ \xi=0$ on $\partial \mathcal{O}^o$ (cf.~\eqref{v-=-v+-xi}). Accordingly, Lemma \ref{lem:smp} (ii) implies that $(\mu_2,\xi)=(0,0)$, and so the proof is complete.
\qquad\end{proof}

By Propositions \ref{prop:uniq} and \ref{prop:J} and by the jump properties of the single layer potential, we deduce the validity of the following theorem on the solution of a mixed problem.
\begin{theorem}\label{thm:ex}
Let $\alpha \in ]0,1[$.  Let $\mathcal{O}^i$, $\mathcal{O}^o$ be bounded open subsets of $\mathbb{R}^n$ of class $C^{1,\alpha}$ such that $\mathcal{O}^o$, $\mathbb{R}^n \setminus \mathrm{cl}\mathcal{O}^i$, and $\mathbb{R}^n \setminus \mathrm{cl}\mathcal{O}^o$ are connected and that $\mathrm{cl}\mathcal{O}^i \subseteq \mathcal{O}^o$. Let $J$ be as in Proposition \ref{prop:J}. Let $(\phi,\gamma) \in C^{0,\alpha}(\partial \mathcal{O}^i)\times C^{1,\alpha}(\partial \mathcal{O}^o)$. Then problem 
\[
\left\{
\begin{array}{ll}
\Delta u(x)=0 & \forall x \in \mathcal{O}^o\setminus \mathrm{cl}\mathcal{O}^i\,,\\
\frac{\partial }{\partial \nu_{\mathcal{O}^i}}u(x)=\phi(x) & \forall x \in  \partial  \mathcal{O}^i\, ,\\
u(x)=\gamma(x) & \forall x \in \partial \mathcal{O}^o\, ,
\end{array}
\right.
\]
has a unique solution $u$ in $C^{1,\alpha}(\mathrm{cl}\mathcal{O}^o \setminus \mathcal{O}^i)$. The solution $u$ is delivered by
\[
u(x)\equiv v[\partial \mathcal{O}^i,\mu_1](x)+v[\partial \mathcal{O}^o,\mu_2](x)+\xi \qquad \forall x \in \mathrm{cl}\mathcal{O}^o \setminus \mathcal{O}^i\, , 
\]
where $(\mu_1,\mu_2,\xi)$ is the unique triple in $C^{0,\alpha}(\partial \mathcal{O}^i)\times C^{0,\alpha}(\partial \mathcal{O}^o)_0\times \mathbb{R}$ such that
\[
J[\mu_1,\mu_2,\xi]=(\phi,\gamma)\, .
\]
\end{theorem}

\section{Formulation of  problem \eqref{bvprho} in terms of integral equations}\label{finteq}

In this section, we formulate problem \eqref{bvprho} in terms of integral equations on $\partial \Omega^i_1$, $\partial\Omega^i_2$, and $\partial \Omega^o$, by exploiting Theorem \ref{thm:ex}  and the rule of change of variables in integrals. Indeed, if $(\varrho_1, \varrho_2) \in \tilde{\mathcal{U}} \cap ]0,+\infty[^2$, by Theorem \ref{thm:ex}, one can convert  problem  \eqref{bvprho} into a system of integral equations which include an equation defined on $\partial \Omega^o$ and two equations defined on the $(\varrho_1,\varrho_2)$-dependent domains $\partial(\varrho_1p^1+\varrho_1 \varrho_2  \Omega^i_1)$ and  $\partial(\varrho_1 p^2+\varrho_1 \varrho_2  \Omega^i_2)$. Then, by exploiting an appropriate change of variable, one  can obtain an equivalent system of integral equations defined on the fixed domains $\partial \Omega^i_1$, $\partial \Omega^i_2$, and $\partial \Omega^o$. 

We find convenient to introduce the following notation. Let $\alpha\in]0,1[$. Let $\Omega^i_1$, $\Omega^i_2$, $\Omega^o$ be as in \eqref{dom}. Let $p^1$, $p^2$ be as in \eqref{p}. Let $r_\ast \in [0,+\infty[$. Let \eqref{assrast} hold. Let $f_1 \in C^{0,\alpha}(\partial \Omega^i_1)$, $f_2 \in C^{0,\alpha}(\partial \Omega^i_2)$,  $g \in C^{1,\alpha}(\partial \Omega^o)$.  Then we introduce the map $\Lambda=(\Lambda_1,\Lambda_2,\Lambda_3)$ from $\tilde{\mathcal{U}}\times C^{0,\alpha}(\partial \Omega^i_1)\times C^{0,\alpha}(\partial \Omega^i_2)\times C^{0,\alpha}(\partial \Omega^o)_0\times \mathbb{R}$ to $C^{0,\alpha}(\partial \Omega^i_1)\times C^{0,\alpha}(\partial \Omega^i_2)\times C^{1,\alpha}(\partial \Omega^o)$ defined by
\[
\begin{split}
\Lambda_1[\varrho_1,&\varrho_2,\theta^{i}_1,\theta^{i}_2,\theta^{o},\xi](t)\equiv\frac{1}{2}\theta^{i}_1(t)+\int_{\partial \Omega^i_1}DS_n(t-s)\nu_{\Omega^i_1}(t)\theta^{i}_1(s)\, d\sigma_s\\ 
& \quad +\varrho_2^{n-1}\int_{\partial \Omega^i_2}DS_n\bigg((p^1-p^2)+\varrho_2(t-s)\bigg)\nu_{\Omega^i_1}(t)\theta^{i}_2(s)\, d\sigma_s\\
& \quad +\int_{\partial \Omega^o}DS_n\big(\varrho_1p^1+\varrho_1\varrho_2 t-y\big)\nu_{\Omega^i_1}(t)\theta^{o}(y)\, d\sigma_y -f_1(t)\qquad \forall t \in \partial \Omega^i_1\, , \\
\Lambda_2[\varrho_1,&\varrho_2,\theta^{i}_1,\theta^{i}_2,\theta^{o},\xi](t)\equiv\frac{1}{2}\theta^{i}_2(t)+\int_{\partial \Omega^i_2}DS_n(t-s)\nu_{\Omega^i_2}(t)\theta^{i}_2(s)\, d\sigma_s\\ 
& \quad +\varrho_2^{n-1}\int_{\partial \Omega^i_1}DS_n\bigg((p^2-p^1)+\varrho_2(t-s)\bigg)\nu_{\Omega^i_2}(t)\theta^{i}_1(s)\, d\sigma_s\\
& \quad +\int_{\partial \Omega^o}DS_n\big(\varrho_1p^2+\varrho_1\varrho_2 t-y\big)\nu_{\Omega^i_2}(t)\theta^{o}(y)\, d\sigma_y -f_2(t)\qquad \forall t \in \partial \Omega^i_2\, ,\\
\Lambda_3[\varrho_1,&\varrho_2,\theta^{i}_1,\theta^{i}_2,\theta^{o},\xi](x)\equiv(\varrho_1 \varrho_2)^{n-1}\sum_{j=1}^2 \int_{\partial \Omega^i_j}S_n(x-\varrho_1p^j-\varrho_1 \varrho_2 s)\theta^{i}_j(s)\, d\sigma_s\\&+\int_{\partial \Omega^o}S_n(x-y)\theta^{o}(y)\, d\sigma_y +\xi-g(x) \qquad \forall x \in \partial \Omega^o\, ,
\end{split}
\]
for all $(\varrho_1,\varrho_2,\theta^{i}_1,\theta^{i}_2,\theta^{o},\xi) \in \tilde{\mathcal{U}}\times C^{0,\alpha}(\partial \Omega^i_1)\times C^{0,\alpha}(\partial \Omega^i_2)\times C^{0,\alpha}(\partial \Omega^o)_0\times \mathbb{R}$.

In the following proposition, we describe the link between the map $\Lambda$ and problem \eqref{bvprho}.

\begin{proposition}\label{prop:finteq}
Let $\alpha\in]0,1[$. Let $\Omega^i_1$, $\Omega^i_2$, $\Omega^o$ be as in \eqref{dom}. Let $p^1$, $p^2$ be as in \eqref{p}. Let $r_\ast \in [0,+\infty[$. Let \eqref{assrast} hold. Let $f_1 \in C^{0,\alpha}(\partial \Omega^i_1)$, $f_2 \in C^{0,\alpha}(\partial \Omega^i_2)$,  $g \in C^{1,\alpha}(\partial \Omega^o)$. Let $(\varrho_1, \varrho_2) \in \tilde{\mathcal{U}} \cap ]0,+\infty[^2$. Then the unique solution $u[\varrho_1,\varrho_2]$ in $C^{1,\alpha}(\mathrm{cl}\Omega(\varrho_1,\varrho_2))$ of problem \eqref{bvprho} is delivered by
\begin{equation}\label{eq:finteq1}
\begin{split}
u[\varrho_1,\varrho_2](x)\equiv &(\varrho_1 \varrho_2)^{n-1}\sum_{j=1}^2 \int_{\partial \Omega^i_j}S_n(x-\varrho_1p^j-\varrho_1\varrho_2 s)\theta^{i}_j[\varrho_1,\varrho_2](s)\, d\sigma_s \\& +\int_{\partial \Omega^o}S_n(x-y)\theta^{o}[\varrho_1,\varrho_2](y)\, d\sigma_y +\xi[\varrho_1,\varrho_2]\qquad \forall x \in \mathrm{cl}\Omega(\varrho_1,\varrho_2)\,,
\end{split}
\end{equation}
where $(\theta^{i}_1[\varrho_1,\varrho_2],\theta^{i}_2[\varrho_1,\varrho_2],\theta^{o}[\varrho_1,\varrho_2],\xi[\varrho_1,\varrho_2])$ is the unique quadruple $(\theta^{i}_1,\theta^{i}_2,\theta^{o},\xi)$ in $C^{0,\alpha}(\partial \Omega^i_1)\times C^{0,\alpha}(\partial \Omega^i_2) \times C^{0,\alpha}(\partial \Omega^o)_0\times \mathbb{R}$ such that
\begin{equation}\label{inteq2a}
\Lambda [\varrho_1,\varrho_2,\theta^{i}_1,\theta^{i}_2,\theta^{o},\xi]=0\, .
\end{equation}
\end{proposition}
\begin{proof} Let $J$ be as in Proposition \ref{prop:J} with
\[
\mathcal{O}^i\equiv \bigg (\varrho_1p^1+\varrho_1\varrho_2 \Omega^i_1\bigg)\cup\bigg(\varrho_1p^2+\varrho_1 \varrho_2\Omega^i_2\bigg) \, , \qquad \mathcal{O}^o\equiv \Omega^o\, .
\]
Then by the definition of $\Lambda$ and the rule of change of variables in integrals one verifies that the quadruple $(\theta^{i}_1,\theta^{i}_2,\theta^{o},\xi)$ in $C^{0,\alpha}(\partial \Omega^i_1)\times C^{0,\alpha}(\partial \Omega^i_2) \times C^{0,\alpha}(\partial \Omega^o)_0\times \mathbb{R}$ is a solution of equation \eqref{inteq2a} if and only if the triple $(\mu_1,\mu_2,\xi)$ in $C^{0,\alpha}(\partial \mathcal{O}^i)\times C^{0,\alpha}(\partial \mathcal{O}^o)_0\times \mathbb{R}$ with $\mu_1$ and $\mu_2$ defined by
\[
\mu_1(x)\equiv
\left\{
\begin{array}{ll}
\theta^i_1\big((x-\varrho_1p^1)/(\varrho_1 \varrho_2)\big) & \forall x \in \varrho_1p^1+\varrho_1 \varrho_2 \partial \Omega_1\, ,   
\\
\theta^i_2\big((x-\varrho_1p^2)/(\varrho_1 \varrho_2)\big) & \forall x \in \varrho_1p^2+\varrho_1 \varrho_2 \partial \Omega_2\, , \end{array}
\right.
\]
\[
\mu_2(x)\equiv \theta^o(x)\qquad \forall x \in \partial \Omega^o\, ,
\]
is a solution of
\[
J[\mu_1,\mu_2,\xi]=(\phi,\gamma)\, ,
\]
with $\phi$ and $\gamma$ defined by
\[
\phi(x)\equiv
\left\{
\begin{array}{ll}
f_1\big((x-\varrho_1 p^1)/(\varrho_1 \varrho_2)\big) & \forall x \in \varrho_1 p^1+\varrho_1 \varrho_2 \partial \Omega_1\, ,   
\\
f_2\big((x-\varrho_1 p^2)/(\varrho_1 \varrho_2)\big) & \forall x \in \varrho_1 p^2+\varrho_1 \varrho_2 \partial \Omega_2\, , \end{array}
\right.
\]
\[
\gamma(x)\equiv g(x)\qquad \forall x \in \partial \Omega^o\, .
\]
Then the conclusion follows by Theorem \ref{thm:ex}. \qquad\end{proof}

By Proposition \ref{prop:finteq}, we are reduced to analyze equation \eqref{inteq2a} around the case $(\varrho_1,\varrho_2)=(0,r_\ast)$. As a first step, in the following lemma we analyze the system which we obtain by taking $(\varrho_1,\varrho_2)=(0,r_\ast)$ in equation \eqref{inteq2a}.

\begin{lemma}\label{lem:lim}
Let $\alpha\in]0,1[$. Let $\Omega^i_1$, $\Omega^i_2$, $\Omega^o$ be as in \eqref{dom}. Let $p^1$, $p^2$ be as in \eqref{p}. Let $r_\ast \in [0,+\infty[$ be such that \eqref{assrast} holds. Let $f_1 \in C^{0,\alpha}(\partial \Omega^i_1)$, $f_2 \in C^{0,\alpha}(\partial \Omega^i_2)$, $g \in C^{1,\alpha}(\partial \Omega^o)$. Then the system of  equations
\begin{equation}\label{inteqlim0}
\begin{split}
&\frac{1}{2}\theta^{i}_1(t)+\int_{\partial \Omega^i_1}DS_n(t-s)\nu_{\Omega^i_1}(t)\theta^{i}_1(s)\, d\sigma_s\\ 
& \quad +r_\ast^{n-1}\int_{\partial \Omega^i_2}DS_n\big((p^1-p^2)+r_\ast (t-s)\big)\nu_{\Omega^i_1}(t)\theta^{i}_2(s)\, d\sigma_s\\
& \quad -\int_{\partial \Omega^o}DS_n(y)\nu_{\Omega^i_1}(t)\theta^{o}(y)\, d\sigma_y -f_1(t)=0\qquad \forall t \in \partial \Omega^i_1\, ,
\end{split}
\end{equation}
\begin{equation}\label{inteqlim1}
\begin{split}
&\frac{1}{2}\theta^{i}_2(t)+\int_{\partial \Omega^i_2}DS_n(t-s)\nu_{\Omega^i_2}(t)\theta^{i}_2(s)\, d\sigma_s\\ 
& \quad +r_\ast^{n-1}\int_{\partial \Omega^i_1}DS_n\big((p^2-p^1)+r_\ast(t-s)\big)\nu_{\Omega^i_2}(t)\theta^{i}_1(s)\, d\sigma_s\\
& \quad -\int_{\partial \Omega^o}DS_n(y)\nu_{\Omega^i_2}(t)\theta^{o}(y)\, d\sigma_y -f_2(t)=0\qquad \forall t \in \partial \Omega^i_2\, ,
\end{split}
\end{equation}
\begin{equation}\label{inteqlim2}
\begin{split}
\int_{\partial \Omega^o}S_n(x-y)\theta^{o}(y)\, d\sigma_y +\xi-g(x)=0 \qquad \forall x \in \partial \Omega^o\, ,
\end{split}
\end{equation}
has a unique solution $(\theta^{i}_1, \theta^{i}_2, \theta^{o}, \xi)$ in $C^{0,\alpha}(\partial \Omega^i_1)\times C^{0,\alpha}(\partial \Omega^i_2)\times C^{0,\alpha}(\partial \Omega^o)_0\times \mathbb{R}$, which we denote by $(\tilde{\theta}^{i}_1, \tilde{\theta}^{i}_2, \tilde{\theta}^{o}, \tilde{\xi})$. 
\end{lemma}
\begin{proof} By Lemma \ref{lem:smp} (ii), equation \eqref{inteqlim2} has a unique solution $(\tilde{\theta}^{o}, \tilde{\xi})$ in the space $C^{0,\alpha}(\partial \Omega^o)_0\times \mathbb{R}$. Then we consider equations \eqref{inteqlim0}, \eqref{inteqlim1} and we introduce the operator $M_{r_\ast}\equiv(M_{r_\ast,1},M_{r_\ast,2})$ from $C^{0,\alpha}(\partial \Omega^i_1)\times C^{0,\alpha}(\partial \Omega^i_2)$ to itself by setting 
\[
\begin{split}
M_{r_\ast,1}&[\theta^{i}_1,\theta^{i}_2](t)\equiv \frac{1}{2}\theta^{i}_1(t)+\int_{\partial \Omega^i_1}DS_n(t-s)\nu_{\Omega^i_1}(t)\theta^{i}_1(s)\, d\sigma_s\\ 
&  +r_\ast^{n-1}\int_{\partial \Omega^i_2}DS_n\big((p^1-p^2)+r_\ast (t-s)\big)\nu_{\Omega^i_1}(t)\theta^{i}_2(s)\, d\sigma_s \qquad \forall t \in \partial \Omega^i_1\, ,
\end{split}
\]
\[
\begin{split}
M_{r_\ast,2}&[\theta^{i}_1,\theta^{i}_2](t)\equiv \frac{1}{2}\theta^{i}_2(t)+\int_{\partial \Omega^i_2}DS_n(t-s)\nu_{\Omega^i_2}(t)\theta^{i}_2(s)\, d\sigma_s\\ 
&  +r_\ast^{n-1}\int_{\partial \Omega^i_1}DS_n\big((p^2-p^1)+r_\ast(t-s)\big)\nu_{\Omega^i_2}(t)\theta^{i}_1(s)\, d\sigma_s \qquad \forall t \in \partial \Omega^i_2\, ,
\end{split}
\]
for all $(\theta^i_1,\theta^i_2) \in C^{0,\alpha}(\partial \Omega^i_1)\times C^{0,\alpha}(\partial \Omega^i_2)$.
We need to show that there exists a unique pair $(\theta^i_1,\theta^i_2)$ such that
\[
\begin{split}
M_{r_\ast,1}&[\theta^{i}_1,\theta^{i}_2](t) =\int_{\partial \Omega^o}DS_n(y)\nu_{\Omega^i_1}(t)\tilde{\theta}^{o}(y)\, d\sigma_y +f_1(t)\qquad \forall t \in \partial \Omega^i_1\, ,
\end{split}
\]
\[
\begin{split}
M_{r_\ast,2}&[\theta^{i}_1,\theta^{i}_2](t) =\int_{\partial \Omega^o}DS_n(y)\nu_{\Omega^i_2}(t)\tilde{\theta}^{o}(y)\, d\sigma_y +f_2(t)\qquad \forall t \in \partial \Omega^i_2\, .
\end{split}
\]
In order to do so, it clearly suffices to show that the operator $M_{r_\ast}$ is invertible. If $r_\ast=0$, the invertibility follows immediately by Lemma \ref{lem:smp} (iii). If $r_\ast > 0$, we note that
\[
\begin{split}
M_{r_\ast,1}&[\theta^{i}_1,\theta^{i}_2]\big((x-p^1)/r_{\ast}\big)= \frac{1}{2}\theta^{i}_1\big((x-p^1)/r_{\ast}\big)\\\ +&\int_{\partial (p^1+r_\ast \Omega^i_1)}DS_n(x-y)\nu_{p^1+r_\ast \Omega^i_1}(x)\theta^{i}_1\big((y-p^1)/r_{\ast}\big)\, d\sigma_y\\ 
 +&\int_{\partial (p^2+r_\ast\Omega^i_2)}DS_n(x-y)\nu_{p^1+r_\ast \Omega^i_1}(x)\theta^{i}_2\big((y-p^2)/r_{\ast}\big)\, d\sigma_y \\ 
 & \qquad \qquad \qquad \qquad \qquad  \qquad \qquad \qquad \forall x \in \partial (p^1+r_\ast \Omega^i_1)\, ,
\end{split}
\]
\[
\begin{split}
M_{r_\ast,2}&[\theta^{i}_1,\theta^{i}_2]\big((x-p^2)/r_{\ast}\big)= \frac{1}{2}\theta^{i}_2\big((x-p^2)/r_{\ast}\big)\\\ +&\int_{\partial (p^2+r_\ast \Omega^i_2)}DS_n(x-y)\nu_{p^2+r_\ast \Omega^i_2}(x)\theta^{i}_2\big((y-p^2)/r_{\ast}\big)\, d\sigma_y\\ 
 +&\int_{\partial (p^1+r_\ast\Omega^i_1)}DS_n(x-y)\nu_{p^2+r_\ast \Omega^i_2}(x)\theta^{i}_1\big((y-p^1)/r_{\ast}\big)\, d\sigma_y \\ 
 & \qquad \qquad \qquad \qquad \qquad  \qquad \qquad \qquad \forall x \in \partial (p^2+r_\ast \Omega^i_2)\, .
\end{split}
\]
As a consequence, the invertibility of $M_{r_\ast}$ follows by Lemma \ref{lem:smp} (iii) with $\Omega\equiv(p^1+r_{\ast}\Omega_1)\cup (p^2+r_{\ast}\Omega_2)$. \qquad\end{proof}
\begin{remark}\label{rem:lim}
Let the assumptions of Lemma \ref{lem:lim} hold. Let $\tilde{u}$ be the unique solution in $C^{1,\alpha}(\mathrm{cl}\Omega^o)$ of 
\[
\left\{
\begin{array}{ll}
\Delta u(x)=0 & \forall x \in \Omega^o\,,\\
u(x)=g(x) & \forall x \in \partial \Omega^o\, .
\end{array}
\right.
\]
Then $\tilde{u}= v^+[\partial \Omega^o,\tilde{\theta}^{o}]+\tilde{\xi}$.
\end{remark}

We are now ready to analyze  equation \eqref{inteq2a} around the degenerate pair $(\varrho_1, \varrho_2)=(0, r_\ast)$.

\begin{proposition}\label{prop:ansol}
Let $\alpha\in]0,1[$. Let $\Omega^i_1$, $\Omega^i_2$, $\Omega^o$ be as in \eqref{dom}. Let $p^1$, $p^2$ be as in \eqref{p}. Let $r_\ast \in [0,+\infty[$. Let \eqref{assrast} hold. Let $\tilde{\mathcal{U}}$ be as in \eqref{tildeU}. Let $f_1 \in C^{0,\alpha}(\partial \Omega^i_1)$, $f_2 \in C^{0,\alpha}(\partial \Omega^i_2)$,  $g \in C^{1,\alpha}(\partial \Omega^o)$. Let $(\tilde{\theta}^{i}_1,\tilde{\theta}^{i}_2, \tilde{\theta}^{o}, \tilde{\xi})$ be as in Lemma \ref{lem:lim}. Then there exist an open neighborhood $\mathcal{U}$ of $(0,r_{\ast})$ in $\mathbb{R}^2$ and a real analytic map $(\Theta^{i}_1,\Theta^{i}_2, \Theta^{o}, \Xi)$ from $\mathcal{U}$ to $C^{0,\alpha}(\partial \Omega^i_1)\times C^{0,\alpha}(\partial \Omega^i_2)\times C^{0,\alpha}(\partial \Omega^o)_0\times \mathbb{R}$ such that
\[
 \mathcal{U} \subseteq \tilde{\mathcal{U}}\, ,
\]
and that
\[
\begin{split}
(\theta^{i}_1[\varrho_1,\varrho_2],\theta^{i}_2[\varrho_1,\varrho_2], \theta^{o}[\varrho_1,\varrho_2], \xi[\varrho_1,\varrho_2])=(\Theta^{i}_1[\varrho_1,\varrho_2],\Theta^{i}_2&[ \varrho_1,\varrho_2], \Theta^{o}[ \varrho_1,\varrho_2], \Xi[ \varrho_1,\varrho_2])\\ &  \forall (\varrho_1,\varrho_2) \in \mathcal{U} \cap ]0,+\infty[^2\, ,
\end{split}
\]
and that
\[
(\tilde{\theta}^{i}_1,\tilde{\theta}^{i}_2, \tilde{\theta}^{o}, \tilde{\xi})=(\Theta^{i}_1[0,r_\ast], \Theta^{i}_2[0,r_\ast], \Theta^{o}[0,r_\ast], \Xi[0,r_\ast])\, .
\]
\end{proposition}
\begin{proof}  By standard properties of integral  operators with real analytic kernels and with no singularity, and by classical mapping properties of layer potentials  (cf.~Miranda~\cite{Mi65}, Lanza de Cristoforis and Rossi \cite[Thm.~3.1]{LaRo04}, Lanza de Cristoforis and the second-named author \cite[\S 4]{LaMu13}), we conclude that $\Lambda$ is real analytic. Now we plan to apply the Implicit Function Theorem to equation $\Lambda[\varrho_1,\varrho_2,\theta^{i}_1,\theta^{i}_2,\theta^{o},\xi]=0$ around the point $(0, r_\ast,\tilde{\theta}^{i}_1,\tilde{\theta}^{i}_2,\tilde{\theta}^{o},\tilde{\xi})$. By definition of $(\tilde{\theta}^{i}_1,\tilde{\theta}^{i}_2,\tilde{\theta}^{o},\tilde{\xi})$, we have $\Lambda[0, r_\ast, \tilde{\theta}^{i}_1,\tilde{\theta}^{i}_2,\tilde{\theta}^{o},\tilde{\xi}]=0$. By standard calculus in Banach spaces, the differential of $\Lambda$ at $(0, r_\ast, \tilde{\theta}^{i}_1,\tilde{\theta}^{i}_2,\tilde{\theta}^{o},\tilde{\xi})$ with respect to the variables $(\theta^{i}_1,\theta^{i}_2,\theta^{o},\xi)$ is delivered by the formulas
\[
\begin{split}
\partial_{(\theta^{i}_1, \theta^{i}_2,\theta^{o},\xi)}\Lambda_1[0,r_\ast,\tilde{\theta}^{i}_1,\tilde{\theta}^{i}_2,\tilde{\theta}^{o},&\tilde{\xi}](\bar{\theta}^{i}_1,\bar{\theta}^{i}_2,\bar{\theta}^{o},\bar{\xi})(t)\\
&\equiv\frac{1}{2}\bar{\theta}^{i}_1(t)+\int_{\partial \Omega^i_1}DS_n(t-s)\nu_{\Omega^i_1}(t)\bar{\theta}^{i}_1(s)\, d\sigma_s\\ 
&  +r_\ast^{n-1}\int_{\partial \Omega^i_2}DS_n\big((p^1-p^2)+r_\ast (t-s)\big)\nu_{\Omega^i_1}(t)\bar{\theta}^{i}_2(s)\, d\sigma_s\\
&  -\int_{\partial \Omega^o}DS_n(y)\nu_{\Omega^i_1}(t)\bar{\theta}^{o}(y)\, d\sigma_y \qquad \forall t \in \partial \Omega^i_1\, ,
\end{split}
\]
\[
\begin{split}
\partial_{(\theta^{i}_1, \theta^{i}_2,\theta^{o},\xi)}\Lambda_2[0,r_\ast,\tilde{\theta}^{i}_1,\tilde{\theta}^{i}_2,\tilde{\theta}^{o},&\tilde{\xi}](\bar{\theta}^{i}_1,\bar{\theta}^{i}_2,\bar{\theta}^{o},\bar{\xi})(t)\\&\equiv\frac{1}{2}\bar{\theta}^{i}_2(t)+\int_{\partial \Omega^i_2}DS_n(t-s)\nu_{\Omega^i_2}(t)\bar{\theta}^{i}_2(s)\, d\sigma_s\\ 
&  +r_\ast^{n-1}\int_{\partial \Omega^i_1}DS_n\big((p^2-p^1)+r_\ast(t-s)\big)\nu_{\Omega^i_2}(t)\bar{\theta}^{i}_1(s)\, d\sigma_s\\
&  -\int_{\partial \Omega^o}DS_n(y)\nu_{\Omega^i_2}(t)\bar{\theta}^{o}(y)\, d\sigma_y\qquad \forall t \in \partial \Omega^i_2\, ,
\end{split}
\]
\[
\begin{split}
\partial_{(\theta^{i}_1, \theta^{i}_2,\theta^{o},\xi)}\Lambda_3[0,r_\ast,\tilde{\theta}^{i}_1,\tilde{\theta}^{i}_2,\tilde{\theta}^{o},\tilde{\xi}](\bar{\theta}^{i}_1,\bar{\theta}^{i}_2,\bar{\theta}^{o},\bar{\xi})(x)\equiv\int_{\partial \Omega^o}S_n(x-y)\bar{\theta}^{o}(y)\, d\sigma_y &+\bar{\xi}\\ &\forall x \in \partial \Omega^o\, ,
\end{split}
\]
for all $(\bar{\theta}^{i}_1,\bar{\theta}^{i}_2,\bar{\theta}^{o},\bar{\xi}) \in C^{0,\alpha}(\partial \Omega^i_1)\times C^{0,\alpha}(\partial \Omega^i_2)\times C^{0,\alpha}(\partial \Omega^o)_0\times \mathbb{R}$. Then, by arguing as in the proof of Lemma \ref{lem:lim}, by classical potential theory, and  by standard calculus in Banach spaces, one can show that $\partial_{(\theta^{i}_1,\theta^{i}_2,\theta^{o},\xi)}\Lambda[0,r_\ast,\tilde{\theta}^{i}_1,\tilde{\theta}^{i}_2,\tilde{\theta}^{o},\tilde{\xi}]$ is a linear homeomorphism from $ C^{0,\alpha}(\partial \Omega^i_1)\times C^{0,\alpha}(\partial \Omega^i_2)\times C^{0,\alpha}(\partial \Omega^o)_0\times \mathbb{R}$ onto $ C^{0,\alpha}(\partial \Omega^i_1)\times C^{0,\alpha}(\partial \Omega^i_2)\times C^{1,\alpha}(\partial \Omega^o)$. Then by the Implicit Function Theorem for real analytic maps in Banach spaces (cf., \textit{e.g.}, Deimling \cite[Theorem 15.3]{De85}), there exist an open neighborhood $\mathcal{U} \subseteq \tilde{\mathcal{U}}$ of $(0,r_{\ast})$ in $\mathbb{R}^2$ and a real analytic map $(\Theta^{i}_1,\Theta^{i}_2, \Theta^{o}, \Xi)$ from $\mathcal{U}$ to $C^{0,\alpha}(\partial \Omega^i_1)\times C^{0,\alpha}(\partial \Omega^i_2)\times C^{0,\alpha}(\partial \Omega^o)_0\times \mathbb{R}$ such that
\begin{equation}\label{eq:ansol}
\Lambda\bigl[\varrho_1, \varrho_2, \Theta^{i}_1[\varrho_1,\varrho_2],\Theta^{i}_2[\varrho_1,\varrho_2], \Theta^{o}[\varrho_1,\varrho_2], \Xi[\varrho_1,\varrho_2]\bigr]=0 \qquad \forall (\varrho_1,\varrho_2) \in  \mathcal{U}\, .
\end{equation}
In particular, by Proposition \ref{prop:finteq} and Lemma \ref{lem:lim}, we have
\[
\begin{split}
(\theta^{i}_1[\varrho_1,\varrho_2],\theta^{i}_2[\varrho_1,\varrho_2], \theta^{o}[\varrho_1,\varrho_2], \xi[\varrho_1,\varrho_2])=(\Theta^{i}_1[\varrho_1,\varrho_2],\Theta^{i}_2[ &\varrho_1,\varrho_2], \Theta^{o}[ \varrho_1,\varrho_2], \Xi[ \varrho_1,\varrho_2])\\ &  \forall (\varrho_1,\varrho_2) \in \mathcal{U} \cap ]0,+\infty[^2\, ,
\end{split}
\]
and
\[
(\tilde{\theta}^{i}_1,\tilde{\theta}^{i}_2, \tilde{\theta}^{o}, \tilde{\xi})=(\Theta^{i}_1[0,r_\ast], \Theta^{i}_2[0,r_\ast], \Theta^{o}[0,r_\ast], \Xi[0,r_\ast])\, ,
\]
and thus the proof is complete. \qquad\end{proof}

\section{A functional analytic representation theorem for the solution of problem  \eqref{bvpe}}
\label{fure}

In the following theorem, we  exploit the analyticity result for the solutions of  equation \eqref{inteq2a} in order to prove representation formulas for $u[\varrho_1,\varrho_2]$ in terms of real analytic maps. Then, by the analysis of the behavior of $u[\varrho_1,\varrho_2]$ for $(\varrho_1,\varrho_2)$ close to the degenerate value $(0,r_\ast)$, we will be able to answer questions (i), (ii), (iii) asked in the introduction and concerning the behavior of the solution $u_\epsilon$ of problem \eqref{bvpe}.  
\begin{theorem}\label{thm:rep}
Let $\alpha\in]0,1[$. Let $\Omega^i_1$, $\Omega^i_2$, $\Omega^o$ be as in \eqref{dom}. Let $p^1$, $p^2$ be as in \eqref{p}. Let $r_\ast \in [0,+\infty[$. Let \eqref{assrast} hold. Let $f_1 \in C^{0,\alpha}(\partial \Omega^i_1)$, $f_2 \in C^{0,\alpha}(\partial \Omega^i_2)$,  $g \in C^{1,\alpha}(\partial \Omega^o)$. Let $\tilde{u}$ be as in Remark \ref{rem:lim}. Let $\mathcal{U}$ be as in Proposition \ref{prop:ansol}. Then the following statements hold.
\begin{enumerate}
\item[(i)] Let $\Omega_M$ be an open subset of $\Omega^o$ such that $0 \not \in \mathrm{cl}\Omega_M$. Then there exist an open neighborhood $\mathcal{U}_{M,\Omega_M}$ of $(0,r_\ast)$ in $\mathbb{R}^2$ and a real analytic map $U_{M,\Omega_M}$ from $\mathcal{U}_{M,\Omega_M}$ to the space $C^{1,\alpha}(\mathrm{cl} \Omega_M)$ such that
\[
\mathcal{U}_{M,\Omega_M} \subseteq \mathcal{U}\, ,  \qquad \mathrm{cl}\Omega_M\subseteq \mathrm{cl}\Omega(\varrho_1,\varrho_2) \qquad \forall (\varrho_1,\varrho_2)\in \mathcal{U}_{M,\Omega_M} \, ,
\]
and such that
\[
\begin{split}
u[\varrho_1,\varrho_2](x)&= U_{M,\Omega_M}[\varrho_1,\varrho_2](x)\qquad\forall x\in \mathrm{cl}\Omega_M\,,
\end{split}
\]
for all $(\varrho_1,\varrho_2)\in\mathcal{U}_{M,\Omega_M} \cap ]0,+\infty[^2$. Moreover, 
\begin{equation}
\begin{split}
\label{eq:rep1}
U_{M,\Omega_M}[0,r_\ast](x)=\tilde{u}(x)\qquad\forall x\in \mathrm{cl}\Omega_M\, .
\end{split}
\end{equation}

\item[(ii)] Let $\Omega_m$ be a bounded open subset of $\mathbb{R}^{n}\setminus\cup_{j=1}^2(p^j+r_\ast{\mathrm{cl}}\Omega^i_j)$. Then there exist an open neighborhood $\mathcal{U}_{m,\Omega_m}$ of $(0,r_\ast)$ in $\mathbb{R}^2$ and a real analytic map $U_{m,\Omega_m}$ from $\mathcal{U}_{m,\Omega_m}$ to the space $C^{1,\alpha}(\mathrm{cl} \Omega_m)$ such that
\[
\mathcal{U}_{m,\Omega_m} \subseteq \mathcal{U}\, , \qquad \varrho_1\mathrm{cl}\Omega_m\subseteq \mathrm{cl}\Omega(\varrho_1,\varrho_2) \qquad \forall (\varrho_1,\varrho_2)\in \mathcal{U}_{m,\Omega_m} \,,
\]
and such that
\[
\begin{split}
u[\varrho_1,\varrho_2](\varrho_1 t)&= U_{m,\Omega_m}[\varrho_1,\varrho_2](t)+\delta_{2,n}\frac{\varrho_1 \varrho_2 \log \varrho_1}{2\pi}\sum_{j=1}^2\int_{\partial \Omega^i_j}f_j \, d\sigma\ \forall t\in \mathrm{cl}\Omega_m\,,
\end{split}
\]
for all $(\varrho_1,\varrho_2)\in\mathcal{U}_{m,\Omega_m} \cap ]0,+\infty[^2$. Moreover, 
\begin{equation}
\begin{split}
\label{eq:rep2}
U_{m,\Omega_m}[0,r_\ast](t)=\tilde{u}(0)\qquad\forall t\in \mathrm{cl}\Omega_m\, .
\end{split}
\end{equation}
\item[(iii)] Let $j \in \{1,2\}$. Let $l \in (\{1,2\}\setminus \{j\})$. Let $\Omega_{m^\ast}$ be a bounded open subset of $\mathbb{R}^{n}\setminus \mathrm{cl}\Omega^i_j$ such that $(p^j +r_\ast \mathrm{cl}\Omega_{m^\ast}) \cap (p^l +r_\ast \mathrm{cl}\Omega^i_l)= \emptyset$. Then there exist an open neighborhood $\mathcal{U}_{m^\ast,\Omega_{m^\ast}}$ of $(0,r_\ast)$ in $\mathbb{R}^2$ and a real analytic map $U_{j,m^\ast,\Omega_{m^\ast}}$ from $\mathcal{U}_{m^\ast,\Omega_{m^\ast}}$ to the space $C^{1,\alpha}(\mathrm{cl} \Omega_{m^\ast})$ such that
\[
\mathcal{U}_{m^\ast,\Omega_{m^\ast}} \subseteq \mathcal{U}\, , \qquad \varrho_1p^j+\varrho_1 \varrho_2\mathrm{cl}\Omega_{m^\ast}\subseteq \mathrm{cl}\Omega(\varrho_1,\varrho_2) \qquad \forall (\varrho_1,\varrho_2)\in \mathcal{U}_{m^\ast,\Omega_{m^\ast}} \,,
\]
and such that
\[
\begin{split}
u[\varrho_1,&\varrho_2](\varrho_1 p^j +\varrho_1\varrho_2 t)= U_{j,m^\ast,\Omega_{m^\ast}}[\varrho_1,\varrho_2](t)\\&+ \delta_{2,n}\varrho_1 \varrho_2 \Bigg(\frac{\log (\varrho_1 \varrho_2)}{2 \pi} \int_{\partial \Omega^i_j}f_j \, d\sigma+ \frac{\ \log \varrho_1 }{2 \pi} \int_{\partial \Omega^i_l}f_l \, d\sigma \Bigg)\quad \forall t \in \mathrm{cl}\Omega_{m^\ast}\,,
\end{split}
\]
for all $(\varrho_1,\varrho_2)\in\mathcal{U}_{m^\ast,\Omega_{m^\ast}} \cap ]0,+\infty[^2$. Moreover, 
\begin{equation}
\begin{split}
\label{eq:rep3}
U_{j,m^\ast,\Omega_{m^\ast}}[0,r_\ast](t)=\tilde{u}(0)\qquad\forall t\in \mathrm{cl}\Omega_{m^\ast}\, .
\end{split}
\end{equation}
\end{enumerate}
(Here the symbol `$M$' stands for `macroscopic' and  the symbols `$m$' and `$m^\ast$' stand for `microscopic'.)
\end{theorem}
\begin{proof}   We first prove statement (i). By possibly taking a bigger $\Omega_M$, we can assume that  $\Omega_M$ is of class $C^1$. Clearly, there exists an open neighborhood $\mathcal{U}_{M,\Omega_M}$ of $(0,r_\ast)$ in $\mathbb{R}^2$ such that $\mathcal{U}_{M,\Omega_M} \subseteq \mathcal{U}$ and that
 \[
 \mathrm{cl}\Omega_M \cap (\cup_{j=1}^2 (\varrho_1 p^j +\varrho_1\varrho_2 \mathrm{cl}\Omega^i_j))= \emptyset \qquad \forall (\varrho_1,\varrho_2) \in \mathcal{U}_{M,\Omega_M}\, .
 \]
Then we introduce the map $U_{M,\Omega_M}$ from $\mathcal{U}_{M,\Omega_M}$ to $C^{1,\alpha}(\mathrm{cl}\Omega_M)$ by setting
\[
\begin{split}
U_{M,\Omega_M}[\varrho_1,\varrho_2](x)&\equiv (\varrho_1 \varrho_2)^{n-1}\sum_{j=1}^2 \int_{\partial \Omega^i_j}S_n(x-\varrho_1p^j-\varrho_1 \varrho_2 s)\Theta^{i}_j[\varrho_1,\varrho_2](s)\, d\sigma_s\\&+\int_{\partial \Omega^o}S_n(x-y)\Theta^{o}[\varrho_1,\varrho_2](y)\, d\sigma_y +\Xi[\varrho_1,\varrho_2] \qquad \forall x \in \mathrm{cl} \Omega_M\, ,
\end{split}
\]
for all $(\varrho_1,\varrho_2) \in \mathcal{U}_{M,\Omega_M}$.  By standard properties of integral  operators with real analytic kernels and with no singularity, by standard properties of functions in Schauder spaces, by classical mapping properties of layer potentials  (cf.~Lanza de Cristoforis and the second-named author \cite{LaMu13},  Miranda~\cite{Mi65}, Lanza de Cristoforis and Rossi \cite[Thm.~3.1]{LaRo04}), and by Proposition \ref{prop:ansol},  we conclude that $U_{M,\Omega_M}$ is real analytic. Moreover, Proposition \ref{prop:ansol} implies that $\Theta^{o}[0,r_\ast]=\tilde{\theta}^{o}$ and that $\Xi^{o}[0,r_\ast]=\tilde{\xi}$, and thus
 \[
 \begin{split}
U_{M,\Omega_M}[0,r_\ast](x)= \int_{\partial \Omega^o}S_n(x-y)\tilde{\theta}^{o}(y)\, d\sigma_y+\tilde{\xi}=\tilde{u}(x) \qquad \forall x \in \mathrm{cl}\Omega_M\, ,
 \end{split}
 \]
 and the validity of equality \eqref{eq:rep1} follows.
 
We now consider statement (ii). By possibly taking a bigger $\Omega_m$, we can assume that  $\Omega_m$ is of class $C^1$. Clearly, there exists an open neighborhood $\mathcal{U}_{m,\Omega_m}$ of $(0,r_\ast)$ in $\mathbb{R}^2$ such that $\mathcal{U}_{m,\Omega_m} \subseteq \mathcal{U}$ and that
 \[
 \mathrm{cl}\Omega_m \cap (\cup_{j=1}^2 ( p^j +\varrho_2 \mathrm{cl}\Omega^i_j))= \emptyset \, , \qquad \varrho_1 \mathrm{cl}\Omega_m \subseteq \mathrm{cl}\Omega^o \qquad \forall (\varrho_1,\varrho_2) \in \mathcal{U}_{M,\Omega_M}\, .
 \]
Then we introduce the map $U_{m,\Omega_m}$ from $\mathcal{U}_{m,\Omega_m}$ to $C^{1,\alpha}(\mathrm{cl}\Omega_m)$ by setting
\[
\begin{split}
U_{m,\Omega_m}[\varrho_1,\varrho_2](t)&\equiv \varrho_1 \varrho_2^{n-1}\sum_{j=1}^2 \int_{\partial \Omega^i_j}S_n(t-p^j- \varrho_2 s)\Theta^{i}_j[\varrho_1,\varrho_2](s)\, d\sigma_s\\&+\int_{\partial \Omega^o}S_n(\varrho_1t-y)\Theta^{o}[\varrho_1,\varrho_2](y)\, d\sigma_y +\Xi[\varrho_1,\varrho_2] \qquad \forall t \in \mathrm{cl} \Omega_m\, ,
\end{split}
\]
for all $(\varrho_1,\varrho_2) \in \mathcal{U}_{m,\Omega_m}$. By equality \eqref{eq:ansol} we have 
\[
\begin{split}
\int_{\partial \Omega^i_j}\Lambda_j\bigl[\varrho_1, \varrho_2, \Theta^{i}_1[\varrho_1, \varrho_2],\Theta^{i}_2[\varrho_1, \varrho_2], \Theta^{o}[\varrho_1, &\varrho_2], \Xi[\varrho_1, \varrho_2]\bigr]\, d\sigma=0  \\& \forall (\varrho_1, \varrho_2) \in \mathcal{U}\, , \ \forall j\in \{1,2\}\, .
\end{split}
\]
Thus, by classical potential theory, we have
\begin{equation}\label{intf}
\int_{\partial \Omega^i_j}\Theta^i_j[\varrho_1, \varrho_2] \, d\sigma= \int_{\partial \Omega^i_j}f_j \, d\sigma \qquad \forall  (\varrho_1, \varrho_2) \in \mathcal{U}\, , \ \forall j \in \{1,2\}\, .
\end{equation}
Then by a simple computation, one verifies that
\[
\begin{split}
u[\varrho_1,\varrho_2](\varrho_1 t)=U_{m,\Omega_m}[\varrho_1,\varrho_2](t)+ \delta_{2,n}\frac{\varrho_1 \varrho_2 \log \varrho_1}{2 \pi}\sum_{j=1}^2 \int_{\partial \Omega^i_j}f_j \, d\sigma \qquad \forall t \in \mathrm{cl}\Omega_m\,,
\end{split}
\] 
for all $(\varrho_1,\varrho_2) \in \mathcal{U}_{m,\Omega_m} \cap ]0,+\infty[^2$. By standard properties of integral  operators with real analytic kernels and with no singularity, by standard properties of functions in Schauder spaces, by classical mapping properties of layer potentials  (cf.~Miranda~\cite{Mi65}, Lanza de Cristoforis and Rossi \cite[Thm.~3.1]{LaRo04}, Lanza de Cristoforis and the second-named author \cite{LaMu13}), and by Proposition \ref{prop:ansol},  we conclude that $U_{m,\Omega_m}$ is real analytic. Moreover, Proposition \ref{prop:ansol} implies that $\Theta^{o}[0,r_\ast]=\tilde{\theta}^{o}$ and that $\Xi^{o}[0,r_\ast]=\tilde{\xi}$, and thus
 \[
 \begin{split}
U_{m,\Omega_m}[0,r_\ast](t)= \int_{\partial \Omega^o}S_n(0-y)\tilde{\theta}^{o}(y)\, d\sigma_y+\tilde{\xi}=\tilde{u}(0) \qquad \forall t \in \mathrm{cl}\Omega_m\, ,
 \end{split}
 \]
 and the validity of equality \eqref{eq:rep2} follows. Thus the proof of statement (ii) is complete. 
 
 We now turn to prove statement (iii). By possibly taking a bigger $\Omega_{m^\ast}$, we can assume that  $\Omega_{m^\ast}$ is of class $C^1$. Clearly, there exists an open neighborhood $\mathcal{U}_{m^\ast,\Omega_{m^\ast}}$ of $(0,r_\ast)$ in $\mathbb{R}^2$ such that $\mathcal{U}_{m^\ast,\Omega_{m^\ast}} \subseteq \mathcal{U}$ and that
 \[
\begin{split}
&\left(p^j+\varrho_2\mathrm{cl}\Omega_{m^\ast}\right)\cap\left(p^l+\varrho_2\mathrm{cl}\Omega^i_l\right)= \emptyset  \qquad \forall (\varrho_1,\varrho_2) \in \mathcal{U}_{m^\ast,\Omega_{m^\ast}}\, ,\\
&\bigg (\varrho_1 p^j+\varrho_1 \varrho_2 \mathrm{cl}\Omega_{m^\ast}\bigg) \subseteq \Omega^o \qquad \forall (\varrho_1,\varrho_2)\in \mathcal{U}_{m^\ast,\Omega_{m^\ast}}\, .
\end{split}
\]
Then we introduce the map $U_{j,m^\ast,\Omega_{m^\ast}}$ from $\mathcal{U}_{m^\ast,\Omega_{m^\ast}}$ to $C^{1,\alpha}(\mathrm{cl}\Omega_{m^\ast})$ by setting
\[
\begin{split}
U_{j,m^\ast,\Omega_{m^\ast}}[&\varrho_1,\varrho_2](t)\equiv \varrho_1 \varrho_2 \int_{\partial \Omega^i_j}S_n(t- s)\Theta^{i}_j[\varrho_1,\varrho_2](s)\, d\sigma_s\\&+\varrho_1\varrho_2^{n-1}\int_{\partial \Omega^i_l}S_n(p^j+\varrho_2 t-p^l -\varrho_2 s)\Theta^{i}_l[\varrho_1,\varrho_2](s)\, d\sigma_s\\&+\int_{\partial \Omega^o}S_n(\varrho_1p^j+\varrho_1 \varrho_2 t-y)\Theta^{o}[\varrho_1,\varrho_2](y)\, d\sigma_y +\Xi[\varrho_1,\varrho_2] \qquad \forall t \in \mathrm{cl} \Omega_{m^\ast}\, ,
\end{split}
\]
for all $(\varrho_1,\varrho_2) \in \mathcal{U}_{m^\ast,\Omega_{m^\ast}}$. By classical potential theory,  by equality \eqref{intf}, and by a simple computation, one verifies that
\[
\begin{split}
u[\varrho_1,\varrho_2](&\varrho_1p^j +\varrho_1 \varrho_2 t)=U_{j,m^\ast,\Omega_{m^\ast}}[\varrho_1,\varrho_2](t)\\
&+ \delta_{2,n}\varrho_1 \varrho_2 \Bigg(\frac{\log (\varrho_1 \varrho_2)}{2 \pi} \int_{\partial \Omega^i_j}f_j \, d\sigma+ \frac{\ \log \varrho_1 }{2 \pi} \int_{\partial \Omega^i_l}f_l \, d\sigma \Bigg)\qquad \forall t \in \mathrm{cl}\Omega_{m^\ast}\,,
\end{split}
\] 
for all $(\varrho_1,\varrho_2) \in \mathcal{U}_{m^\ast,\Omega_{m^\ast}} \cap ]0,+\infty[^2$. By standard properties of integral  operators with real analytic kernels and with no singularity, by standard properties of functions in Schauder spaces, by classical mapping properties of layer potentials  (cf.~Lanza de Cristoforis and the second-named author \cite{LaMu13},  Miranda~\cite{Mi65}, Lanza de Cristoforis and Rossi \cite[Thm.~3.1]{LaRo04}), and by Proposition \ref{prop:ansol},  we conclude that $U_{j,m^\ast,\Omega_{m^\ast}}$ is real analytic. Moreover, Proposition \ref{prop:ansol} implies that $\Theta^{o}[0,r_\ast]=\tilde{\theta}^{o}$ and that $\Xi^{o}[0,r_\ast]=\tilde{\xi}$, and thus
 \[
 \begin{split}
U_{j,m^\ast,\Omega_{m^\ast}}[0,r_\ast](t)= \int_{\partial \Omega^o}S_n(0-y)\tilde{\theta}^{o}(y)\, d\sigma_y+\tilde{\xi}=\tilde{u}(0) \qquad \forall t \in \mathrm{cl}\Omega_{m^\ast}\, ,
 \end{split}
 \]
 and the validity of equality \eqref{eq:rep3} follows. \qquad\end{proof}
 
Then by Theorem \ref{thm:rep}, we immediately deduce the validity of the following.

\begin{corollary}\label{cor:repeps}
Let the assumptions of Theorem \ref{thm:rep} hold. Let $\eta$, $r_\ast$ be as in \eqref{eta}. Let $\epsilon_{0}$ be as in \eqref{e0}.  Then the following statements hold.
\begin{enumerate}
\item[(i)] Let $\Omega_M$, $\mathcal{U}_{M,\Omega_M}$, $U_{M,\Omega_M}$ be as in Theorem \ref{thm:rep} (i). Then there exists
$\epsilon_{M,\Omega_M}\in]0,\epsilon_0[$ such that
\[
(\eta(\epsilon),\epsilon/\eta(\epsilon)) \in \mathcal{U}_{M,\Omega_M}\, , \qquad \mathrm{cl}\Omega_M\subseteq \mathrm{cl}\Omega_\epsilon \qquad \forall \epsilon\in ]0,\epsilon_{M,\Omega_M}[ \,,
\]
and such that
\[
\begin{split}
u_\epsilon(x)&= U_{M,\Omega_M}[\eta(\epsilon),\epsilon/\eta(\epsilon)](x)\qquad\forall x\in \mathrm{cl}\Omega_M\,,
\end{split}
\]
for all $\epsilon\in]0,\epsilon_{M,\Omega_M}[$. 
\item[(ii)] Let $\Omega_m$, $\mathcal{U}_{m,\Omega_m}$, $U_{m,\Omega_m}$ be as in Theorem \ref{thm:rep} (ii). Then there exists $\epsilon_{m,\Omega_m}\in]0,\epsilon_0[$ such that
\[
(\eta(\epsilon),\epsilon/\eta(\epsilon)) \in \mathcal{U}_{m,\Omega_m}\, , \qquad \eta(\epsilon)\mathrm{cl}\Omega_m\subseteq \mathrm{cl}\Omega_\epsilon \qquad \forall \epsilon\in ]0,\epsilon_{m,\Omega_m}[ \,,
\]
and such that
\[
\begin{split}
u_\epsilon(\eta(\epsilon)t)&= U_{m,\Omega_m}[\eta(\epsilon),\epsilon/\eta(\epsilon)](t)+\delta_{2,n}\frac{\epsilon \log \eta(\epsilon)}{2\pi}\sum_{j=1}^2\int_{\partial \Omega^i_j}f_j \, d\sigma\quad\forall t\in \mathrm{cl}\Omega_m\,,
\end{split}
\]
for all $\epsilon\in]0,\epsilon_{m,\Omega_m}[$.
\item[(iii)] Let $j$, $l$, $\Omega_{m^\ast}$, $\mathcal{U}_{m^\ast,\Omega_{m^\ast}}$, $U_{j,m^\ast,\Omega_{m^\ast}}$ be as in Theorem \ref{thm:rep} (iii). Then there exists $\epsilon_{m^\ast,\Omega_{m^\ast}}\in]0,\epsilon_0[$ such that
\[
(\eta(\epsilon),\epsilon/\eta(\epsilon)) \in \mathcal{U}_{m^\ast,\Omega_{m^\ast}}\, , \qquad \eta(\epsilon)p^j+\epsilon\mathrm{cl}\Omega_{m^\ast}\subseteq \mathrm{cl}\Omega_\epsilon \qquad \forall \epsilon\in ]0,\epsilon_{m^\ast,\Omega_{m^\ast}}[ \,,
\]
and such that
\[
\begin{split}
u_\epsilon(&\eta(\epsilon) p^j +\epsilon t)= U_{j,m^\ast,\Omega_{m^\ast}}[\eta(\epsilon),\epsilon/\eta(\epsilon)](t)\\&+ \delta_{2,n}\epsilon \Bigg(\frac{\log \epsilon}{2 \pi} \int_{\partial \Omega^i_j}f_j \, d\sigma+ \frac{\ \log \eta(\epsilon) }{2 \pi} \int_{\partial \Omega^i_l}f_l \, d\sigma \Bigg)\qquad \forall t \in \mathrm{cl}\Omega_{m^\ast}\,,
\end{split}
\]
for all $\epsilon\in ]0,\epsilon_{m^\ast,\Omega_{m^\ast}}[$. \end{enumerate}
\end{corollary}

\begin{remark}
Under the assumptions of Corollary \ref{cor:repeps}, we note that if $x \in \mathrm{cl}\Omega^o \setminus \{0\}$ is fixed, then we can deduce the existence of a sequence $\{c_{(j_1,j_2)}\}_{(j_1,j_2)\in \mathbb{N}^2\setminus \{(0,0)\}}$ such that
\[
u_\epsilon(x)=\tilde{u}(x)+\sum_{(j_1,j_2)\in \mathbb{N}^2\setminus \{(0,0)\}}c_{(j_1,j_2)} \Big(\eta(\epsilon)\Big)^{j_1}\bigg(\frac{\epsilon}{\eta(\epsilon)}-r_\ast\bigg)^{j_2} \, ,
\]
for $\epsilon$ in a neighborhood of $0$. Moreover, if we know that $\eta(\epsilon)$ equals the restriction to positive values of $\epsilon$ of a real analytic function defined in a neighborhood of $0$, then by \eqref{eta} the function $\epsilon/\eta(\epsilon)$ has a real analytic continuation in a neighborhood of $\epsilon=0$ and thus we can deduce the existence of a sequence $\{c_j\}_{j \in \mathbb{N}\setminus \{0\}}$ such that
\[
u_\epsilon(x)=\tilde{u}(x)+\sum_{j \in \mathbb{N}\setminus \{0\}}c_j \epsilon^j \, ,
\]
for $\epsilon$ small and positive, where the series converges absolutely in a neighborhood of $0$.
\end{remark}

\section{Asymptotic expansion of the solution of the mixed problem}\label{asy} 

The aim of this section is to provide an asymptotic expansion of the solution $u[\varrho_1,\varrho_2]$ of the mixed problem \eqref{bvprho} as $(\varrho_1,\varrho_2)$ tends to the degenerate value $(0,r_\ast)$. We shall assume that $r_\ast =0$ and we will focus on the two-dimensional case. As already done in \cite{DaMuRo} for the Dirichlet problem for the Laplace equation, since the solution is represented by means of layer potentials, we first need to obtain expansions of the densities of the layer potentials. Therefore, here we first compute an expansion in the variable $(\varrho_1,\varrho_2)$ of $(\Theta^{i}_1[\varrho_1,\varrho_2],\Theta^{i}_2[ \varrho_1,\varrho_2], \Theta^{o}[ \varrho_1,\varrho_2], \Xi[ \varrho_1,\varrho_2])$ for $(\varrho_1,\varrho_2)$ close to the degenerate value $(0,r_\ast)=(0,0)$. On the other hand, by the real analyticity of $(\Theta^{i}_1,\Theta^{i}_2, \Theta^{o}, \Xi)$ (cf.~Proposition \ref{prop:ansol}), we know that there exist families $\{\theta^{i}_{1,(j,k)}\}_{(j,k) \in \mathbb{N}^2} \subseteq C^{0,\alpha}(\partial \Omega^i_1)$,  $\{\theta^{i}_{2,(j,k)}\}_{(j,k)\in \mathbb{N}^2} \subseteq C^{0,\alpha}(\partial \Omega^i_2)$, $\{\theta^{o}_{(j,k)}\}_{(j,k) \in \mathbb{N}^2} \subseteq C^{0,\alpha}(\partial \Omega^o)_0$, $\{\xi_{(j,k)}\}_{(j,k) \in \mathbb{N}^2} \subseteq \mathbb{R}$, such that for  $(\varrho_1,\varrho_2)$ in a neighborhood of $(0,0)$ we have
\[
\begin{split}
& \Theta^{i}_1[\varrho_1,\varrho_2]=\sum_{(j,k)\in \mathbb{N}^2} \frac{\theta^{i}_{1,(j,k)}}{j!k!}\varrho_1^j \varrho_2^k\, , \qquad \Theta^{i}_2[\varrho_1,\varrho_2]=\sum_{(j,k)\in \mathbb{N}^2} \frac{\theta^{i}_{2,(j,k)}}{j!k!}\varrho_1^j \varrho_2^k\, , \\
& \Theta^{o}[\varrho_1,\varrho_2]=\sum_{(j,k)\in \mathbb{N}^2} \frac{\theta^{o}_{(j,k)}}{j!k!}\varrho_1^j \varrho_2^k\, , \qquad \  \ \Xi[\varrho_1,\varrho_2]=\sum_{(j,k)\in \mathbb{N}^2} \frac{\xi_{(j,k)}}{j!k!}\varrho_1^j \varrho_2^k\, ,
\end{split}
\]
where the series converge absolutely in $C^{0,\alpha}(\partial \Omega^i_1)$, in $C^{0,\alpha}(\partial \Omega^i_2)$, in $C^{0,\alpha}(\partial \Omega^o)_0$, and in $\mathbb{R}$, respectively, uniformly for $(\varrho_1,\varrho_2)$ in a compact neighborhood of $(0,0)$. In particular,
\[
\begin{split}
&\theta^{i}_{1,(j,k)}=\partial_{\varrho_1}^j \partial_{\varrho_2}^k \Theta^{i}_1[0,0]\, , \qquad  \theta^{i}_{2,(j,k)}=\partial_{\varrho_1}^j \partial_{\varrho_2}^k \Theta^{i}_2[0,0]\, ,\\
& \theta^{o}_{(j,k)}=\partial_{\varrho_1}^j \partial_{\varrho_2}^k \Theta^{o}[0,0]\, , \qquad \ \ \xi_{(j,k)}=\partial_{\varrho_1}^j \partial_{\varrho_2}^k \Xi[0,0]\,, 
\end{split}
\]
for all $(j,k) \in \mathbb{N}^2 \setminus \{(0,0)\}$, and
\[
\begin{split}
&\theta^{i}_{1,(0,0)}=\tilde{\theta}^{i}_{1}\, , \qquad  \theta^{i}_{2,(0,0)}=\tilde{\theta}^{i}_{2}\, ,\\
& \theta^{o}_{(0,0)}=\tilde{\theta}^{o}\, , \qquad \ \ \xi_{(0,0)}=\tilde{\xi}\, .
\end{split}
\]
We now plan to identify some suitable coefficients $\theta^{i}_{1,(j,k)}$,  $\theta^{i}_{2,(j,k)}$, $\theta^{o}_{(j,k)}$, $\xi_{(j,k)}$ as the solutions of certain integral equations, in order to study the asymptotic expansion of $u[\varrho_1, \varrho_2]$. To do so, we shall exploit the fact that by equality \eqref{eq:ansol} we have 
\begin{equation}\label{eq:ansol:der}
\begin{split}
\partial_{\varrho_1}^j \partial_{\varrho_2}^k\Lambda\bigl[\varrho_1, \varrho_2, \Theta^{i}_1[\varrho_1,\varrho_2],\Theta^{i}_2[\varrho_1,\varrho_2], \Theta^{o}[\varrho_1,&\varrho_2], \Xi[\varrho_1,\varrho_2]\bigr]=0 \\ &\forall (\varrho_1,\varrho_2) \in  \mathcal{U}\, , \forall (j,k) \in \mathbb{N}^2\, .
\end{split}
\end{equation}

In the following lemma we consider the first coefficients $\theta^o_{(j,k)}$, $\xi_{(j,k)}$. In particular, we show that if $n=2$, then $\theta^o_{(j,0)}$,  $\theta^o_{(0,k)}$, $\xi_{(j,0)}$, and $\xi_{(0,k)}$ are all equal to $0$  for all $(j,k) \in  \mathbb{N}^2\setminus \{(0,0)\}$.

\begin{lemma}\label{lem:1}
Let $r_\ast=0$. Let the assumptions of Proposition \ref{prop:ansol} hold.  Then
\[
\begin{split}
\theta^o_{(j,k)}=0\, , \ \xi_{(j,k)}=0\, ,
\end{split}
\]
for all $(j,k) \in \Big(\{0,1,\dots,n-2\}\times \big(\mathbb{N} \setminus \{0\}\big) \Big)\cup  \Big(  \big(\mathbb{N} \setminus \{0\}\big) \times \{0,1,\dots,n-2\}\Big)$.
In particular, if $n=2$, then 
\[
\theta^o_{(j,0)}=0\, ,\ \theta^o_{(0,k)}=0\, , \ \xi_{(j,0)}=0\, , \ \xi_{(0,k)}=0\, , \qquad \forall (j,k) \in  \mathbb{N}^2\setminus \{(0,0)\}\, .
\]
\end{lemma}
 \begin{proof}  Let $(j,k) \in \{0,1,\dots,n-2\}\times (\mathbb{N} \setminus \{0\})$. A simple computation shows that
 \begin{equation}\label{eq:1:4}
 \begin{split}
 \partial_{\varrho_1}^j &\partial_{\varrho_2}^k \Lambda_3\bigl[\varrho_1, \varrho_2, \Theta^{i}_1[\varrho_1,\varrho_2],\Theta^{i}_2[\varrho_1,\varrho_2], \Theta^{o}[\varrho_1,\varrho_2], \Xi[\varrho_1,\varrho_2]\bigr](x)\\
 &=\int_{\partial \Omega^o}S_n(x-y)\partial_{\varrho_1}^j \partial_{\varrho_2}^k\Theta^{o}[\varrho_1,\varrho_2](y)\, d\sigma_y +\partial_{\varrho_1}^j \partial_{\varrho_2}^k\Xi[\varrho_1,\varrho_2]\\&+\partial_{\varrho_1}^j \bigg(\varrho_1^{n-1}\tilde{R}_1[\varrho_1,\varrho_2](x)\bigg) \qquad \forall x \in \partial \Omega^o\, ,
 \end{split}
 \end{equation}
for $(\varrho_1,\varrho_2) \in  \mathcal{U}$, where $\tilde{R}_1$ is a real analytic function from $\mathcal{U}$ to $C^{1,\alpha}(\partial \Omega^o)$. Accordingly, by \eqref{eq:ansol:der} and \eqref{eq:1:4}, we have
\[
\begin{split}
0=\int_{\partial \Omega^o}S_n(x-y)\partial_{\varrho_1}^j \partial_{\varrho_2}^k\Theta^{o}[\varrho_1,\varrho_2](y)\, d\sigma_y +\partial_{\varrho_1}^j \partial_{\varrho_2}^k\Xi[\varrho_1,\varrho_2]+\varrho_1^{n-1-j}&\tilde{R}_2[\varrho_1,\varrho_2](x) \\& \forall x \in \partial \Omega^o\, ,
\end{split}
\]
 for $(\varrho_1,\varrho_2) \in  \mathcal{U}$, where  $\tilde{R}_2$ is a real analytic function from $\mathcal{U}$ to $C^{1,\alpha}(\partial \Omega^o)$. Then, by taking $(\varrho_1,\varrho_2)=(0,0)$ we obtain
 \[
 0=\int_{\partial \Omega^o}S_n(x-y)\partial_{\varrho_1}^j \partial_{\varrho_2}^k\Theta^{o}[0,0](y)\, d\sigma_y +\partial_{\varrho_1}^j \partial_{\varrho_2}^k\Xi[0,0] \qquad \forall x \in \partial \Omega^o\, ,
 \]
 which implies 
 \[
 \partial_{\varrho_1}^j \partial_{\varrho_2}^k\Theta^{o}[0,0]=0\, ,\qquad \partial_{\varrho_1}^j \partial_{\varrho_2}^k\Xi[0,0]=0\, ,
 \]
 \textit{i.e.},
  \[
\theta^{o}_{(j,k)}=0\, ,\qquad \xi_{(j,k)}=0\, .
 \]
 Similarly, one shows that if $(j,k) \in (\mathbb{N}\setminus \{0\}) \times \{0,1,\dots,n-2\}$, then
 \[
\theta^{o}_{(j,k)}=0\, ,\qquad \xi_{(j,k)}=0\, 
 \]
(cf. Lemma \ref{lem:smp} (ii)). \qquad\end{proof}

We now confine ourselves to the case $n=2$. In Lemmas \ref{lem:2} and \ref{lem:3} below, we provide the integral equations which identify the functions $\theta^{i}_{1,(1,0)}$, $\theta^{i}_{2,(1,0)}$, $\theta^{i}_{1,(0,1)}$, and $\theta^{i}_{2,(0,1)}$.

\begin{lemma}\label{lem:2}
Let $n=2$. Let $r_\ast=0$. Let the assumptions of Proposition \ref{prop:ansol} hold.  Then $\theta^{i}_{1,(1,0)}$ is the unique function in $C^{0,\alpha}(\partial \Omega^i_1)$ such that
 \begin{equation}\label{eq:2:5}
 \begin{split}
 &\frac{1}{2} \theta^{i}_{1,(1,0)}(t)+\int_{\partial \Omega^i_1}DS_2(t-s)\nu_{\Omega^i_1}(t) \theta^{i}_{1,(1,0)}(s)\, d\sigma_s\\ 
&  +\sum_{h,k=1}^2(p^1)_h(\nu_{\Omega^i_1}(t))_k\int_{\partial \Omega^o}\big(\partial_h\partial_kS_2\big)(y)\theta^{o}_{(0,0)}(y)\, d\sigma_y=0\qquad \forall t \in \partial \Omega^i_1\, ,
\end{split}
 \end{equation}
 and $\theta^{i}_{2,(1,0)}$ is the unique function in $C^{0,\alpha}(\partial \Omega^i_2)$ such that

\[
 \begin{split}
 &\frac{1}{2} \theta^{i}_{2,(1,0)}(t)+\int_{\partial \Omega^i_2}DS_2(t-s)\nu_{\Omega^i_2}(t) \theta^{i}_{2,(1,0)}(s)\, d\sigma_s\\ 
&  +\sum_{h,k=1}^2(p^2)_h(\nu_{\Omega^i_2}(t))_k\int_{\partial \Omega^o}\big(\partial_h\partial_kS_2\big)(y)\theta^{o}_{(0,0)}(y)\, d\sigma_y=0\qquad \forall t \in \partial \Omega^i_2\, .
\end{split}
\]
Moreover,
\[
\int_{\partial \Omega^i_1}\theta^{i}_{1,(1,0)}\, d\sigma=0\, , \qquad \int_{\partial \Omega^i_2}\theta^{i}_{2,(1,0)}\, d\sigma=0\, .
\]
\end{lemma}
 \begin{proof}  If $(\varrho_1,\varrho_2) \in \mathcal{U}$, then by differentiating 
 \[
 \Lambda_1\bigl[\varrho_1, \varrho_2, \Theta^{i}_1[\varrho_1,\varrho_2],\Theta^{i}_2[\varrho_1,\varrho_2], \Theta^{o}[\varrho_1,\varrho_2], \Xi[\varrho_1,\varrho_2]\bigr]
 \]
 for $n=2$, we deduce that
\begin{equation}\label{eq:2:1}
\begin{split}
 \partial_{\varrho_1} &\Lambda_1\bigl[\varrho_1, \varrho_2, \Theta^{i}_1[\varrho_1,\varrho_2],\Theta^{i}_2[\varrho_1,\varrho_2], \Theta^{o}[\varrho_1,\varrho_2], \Xi[\varrho_1,\varrho_2]\bigr](t)\\
 &=\frac{1}{2} \partial_{\varrho_1}\Theta^{i}_1[\varrho_1,\varrho_2](t)+\int_{\partial \Omega^i_1}DS_2(t-s)\nu_{\Omega^i_1}(t) \partial_{\varrho_1}\Theta^{i}_1[\varrho_1,\varrho_2](s)\, d\sigma_s\\ 
&  +\varrho_2\int_{\partial \Omega^i_2}DS_2\bigg((p^1-p^2)+\varrho_2(t-s)\bigg)\nu_{\Omega^i_1}(t) \partial_{\varrho_1}\Theta^{i}_2[\varrho_1,\varrho_2](s)\, d\sigma_s\\
&  +\sum_{h,k=1}^2\int_{\partial \Omega^o}\bigg[\Big(\partial_h\partial_kS_2\Big)\big(\varrho_1p^1+\varrho_1\varrho_2 t-y\big)\bigg](p^1+\varrho_2 t)_h(\nu_{\Omega^i_1}(t))_k\Theta^{o}[\varrho_1,\varrho_2](y)\, d\sigma_y\\
&  +\int_{\partial \Omega^o}DS_2\big(\varrho_1p^1+\varrho_1\varrho_2 t-y\big)\nu_{\Omega^i_1}(t)\partial_{\varrho_1}\Theta^{o}[\varrho_1,\varrho_2](y)\, d\sigma_y\qquad \forall t \in \partial \Omega^i_1\, . 
\end{split}
\end{equation}
Then by equality \eqref{eq:ansol:der}, by formula \eqref{eq:2:1},  by taking $(\varrho_1,\varrho_2)=(0,0)$, by Lemma \ref{lem:1}, and by classical potential theory (see also Lemma \ref{lem:smp} (iii)), we deduce that $\theta^{i}_{1,(1,0)}$ is the unique function in $C^{0,\alpha}(\partial \Omega^i_1)$ such that equation \eqref{eq:2:5} holds. By integrating equality \eqref{eq:2:5}, we also deduce that $\int_{\partial \Omega^i_1}\theta^{i}_{1,(1,0)}\, d\sigma=0$. Similarly, one argues for $\theta^i_{2,(1,0)}$. \qquad\end{proof}

\begin{lemma}\label{lem:3}
Let $n=2$. Let $r_\ast=0$. Let the assumptions of Proposition \ref{prop:ansol} hold. Then $\theta^{i}_{1,(0,1)}$ is the unique function in $C^{0,\alpha}(\partial \Omega^i_1)$ such that 
\[
\begin{split}
\frac{1}{2} \theta^{i}_{1,(0,1)}(t)+&\int_{\partial \Omega^i_1}DS_2(t-s)\nu_{\Omega^i_1}(t)  \theta^{i}_{1,(0,1)}(s)\, d\sigma_s\\ 
& =DS_2\big(p^2-p^1\big)\nu_{\Omega^i_1}(t)\int_{\partial \Omega^i_2} f_2\, d\sigma \qquad \forall t \in \partial \Omega^i_1\, ,
\end{split}
\]
and $\theta^{i}_{2,(0,1)}$ is the unique function in $C^{0,\alpha}(\partial \Omega^i_2)$ such that
\begin{equation}\label{eq:3:5}
\begin{split}
\frac{1}{2} \theta^{i}_{2,(0,1)}(t)+&\int_{\partial \Omega^i_2}DS_2(t-s)\nu_{\Omega^i_2}(t)  \theta^{i}_{2,(0,1)}(s)\, d\sigma_s\\ 
& =DS_2\big(p^1-p^2\big)\nu_{\Omega^i_2}(t)\int_{\partial \Omega^i_1} f_1\, d\sigma  \qquad \forall t \in \partial \Omega^i_2\, .
\end{split}
\end{equation}
In particular,
\[
\int_{\partial \Omega^i_1}\theta^{i}_{1,(0,1)}\, d\sigma=0\, , \qquad \int_{\partial \Omega^i_2}\theta^{i}_{2,(0,1)}\, d\sigma=0\, .
\]
\end{lemma}
 \begin{proof}  If $(\varrho_1,\varrho_2) \in \mathcal{U}$, then by differentiating 
 \[
 \Lambda_1\bigl[\varrho_1, \varrho_2, \Theta^{i}_1[\varrho_1,\varrho_2],\Theta^{i}_2[\varrho_1,\varrho_2], \Theta^{o}[\varrho_1,\varrho_2], \Xi[\varrho_1,\varrho_2]\bigr]
 \] 
 for $n=2$, we deduce that
\begin{equation}\label{eq:3:1}
\begin{split}
 &\partial_{\varrho_2} \Lambda_1\bigl[\varrho_1, \varrho_2, \Theta^{i}_1[\varrho_1,\varrho_2],\Theta^{i}_2[\varrho_1,\varrho_2], \Theta^{o}[\varrho_1,\varrho_2], \Xi[\varrho_1,\varrho_2]\bigr](t)\\
 &=\frac{1}{2} \partial_{\varrho_2}\Theta^{i}_1[\varrho_1,\varrho_2](t)+\int_{\partial \Omega^i_1}DS_2(t-s)\nu_{\Omega^i_1}(t) \partial_{\varrho_2}\Theta^{i}_1[\varrho_1,\varrho_2](s)\, d\sigma_s\\ 
&  +\int_{\partial \Omega^i_2}DS_2\bigg((p^1-p^2)+\varrho_2(t-s)\bigg)\nu_{\Omega^i_1}(t) \Theta^{i}_2[\varrho_1,\varrho_2](s)\, d\sigma_s\\
&  +\varrho_2\sum_{h,k=1}^2(\nu_{\Omega^i_1}(t))_h\int_{\partial \Omega^i_2}\bigg[\Big(\partial_h\partial_k S_2\Big)\bigg((p^1-p^2)+\varrho_2(t-s)\bigg)\bigg] (t-s)_k\Theta^{i}_2[\varrho_1,\varrho_2](s)\, d\sigma_s\\
&  +\varrho_2 \int_{\partial \Omega^i_2}DS_2\bigg((p^1-p^2)+\varrho_2(t-s)\bigg)\nu_{\Omega^i_1}(t) \partial_{\varrho_2}\Theta^{i}_2[\varrho_1,\varrho_2](s)\, d\sigma_s\\
&+\varrho_1 \sum_{h,k=1}^{2}t_h(\nu_{\Omega^i_1}(t))_k\int_{\partial \Omega^o}\bigg[\Big(\partial_h\partial_kS_2\Big)\big(\varrho_1p^1+\varrho_1\varrho_2 t-y\big)\bigg]\Theta^{o}[\varrho_1,\varrho_2](y)\, d\sigma_y\\
&  +\int_{\partial \Omega^o}DS_2\big(\varrho_1p^1+\varrho_1\varrho_2 t-y\big)\nu_{\Omega^i_1}(t)\partial_{\varrho_2}\Theta^{o}[\varrho_1,\varrho_2](y)\, d\sigma_y\qquad \forall t \in \partial \Omega^i_1\, . 
\end{split}
\end{equation}
Then by equality \eqref{eq:ansol:der}, by formula \eqref{eq:3:1},  by taking $(\varrho_1,\varrho_2)=(0,0)$, by Lemma \ref{lem:1}, and by classical potential theory (see also Lemma \ref{lem:smp} (iii)), we deduce that $\theta^{i}_{1,(0,1)}$ is the unique function in $C^{0,\alpha}(\partial \Omega^i_1)$ such that equation \eqref{eq:3:5} holds. By integrating equality \eqref{eq:3:5}, we also deduce that $\int_{\partial \Omega^i_1}\theta^{i}_{1,(0,1)}\, d\sigma=0$. Analogously, one proceeds for $\theta^i_{2,(0,1)}$. \qquad\end{proof}

\begin{remark}\label{rem:2}
Let $n=2$. Let $r_\ast=0$. Let the assumptions of Proposition \ref{prop:ansol} hold. By arguing as in the proof of Lemma \ref{lem:3}, one shows that
\[
\begin{split}
 \partial_{\varrho_2} &\Lambda_3\bigl[\varrho_1, \varrho_2, \Theta^{i}_1[\varrho_1,\varrho_2],\Theta^{i}_2[\varrho_1,\varrho_2], \Theta^{o}[\varrho_1,\varrho_2], \Xi[\varrho_1,\varrho_2]\bigr](x)\\
&= \int_{\partial \Omega^o}S_2(x-y) \partial_{\varrho_2}\Theta^{o}[\varrho_1,\varrho_2](y)\, d\sigma_y + \partial_{\varrho_2}\Xi[\varrho_1,\varrho_2]\\
 &+\varrho_1 \sum_{j=1}^2\int_{\partial \Omega^i_j}S_2(x-\varrho_1p^j-\varrho_1 \varrho_2 s)\Theta^{i}_j[\varrho_1,\varrho_2](s)\, d\sigma_s\\
&-\varrho_1^2 \varrho_2\sum_{h,j=1}^2 \int_{\partial \Omega^i_j}\bigg[\Big(\partial_hS_2\Big)(x-\varrho_1p^j-\varrho_1 \varrho_2 s)\Big]s_h\Theta^{i}_j[\varrho_1,\varrho_2](s)\, d\sigma_s\\
&+\varrho_1 \varrho_2\sum_{j=1}^2 \int_{\partial \Omega^i_j}S_2(x-\varrho_1p^j-\varrho_1 \varrho_2 s)\partial_{\varrho_2}\Theta^{i}_j[\varrho_1,\varrho_2](s)\, d\sigma_s  \qquad \forall x \in \partial \Omega^o\, ,
\end{split}
\]
for all $(\varrho_1,\varrho_2)\in \mathcal{U}$.
\end{remark}

In the following lemma, instead, we consider $\theta^{o}_{(1,1)}$ and $\xi_{(1,1)}$. 

\begin{lemma}\label{lem:4}
Let $n=2$. Let $r_\ast=0$. Let the assumptions of Proposition \ref{prop:ansol} hold.  Then $(\theta^{o}_{(1,1)},\xi_{(1,1)})$ is the unique pair in $C^{0,\alpha}(\partial \Omega^o)_0\times \mathbb{R}$ such that 
\begin{equation}\label{eq:4:6}
\int_{\partial \Omega^o}S_2(x-y)\theta^o_{(1,1)}(y)\, d\sigma_y +\xi_{(1,1)}=-S_2(x)\sum_{j=1}^2\int_{\partial \Omega^i_j}f_j\,d\sigma \qquad \forall x \in \partial \Omega^o\, .
\end{equation}
\end{lemma}
\begin{proof}  If $(\varrho_1,\varrho_2) \in \mathcal{U}$, then by differentiating 
\[ 
\partial_{\varrho_2}\Lambda_3\bigl[\varrho_1, \varrho_2, \Theta^{i}_1[\varrho_1,\varrho_2],\Theta^{i}_2[\varrho_1,\varrho_2], \Theta^{o}[\varrho_1,\varrho_2], \Xi[\varrho_1,\varrho_2]\bigr]
\] 
for $n=2$ (cf.~Remark \ref{rem:2}), we deduce that
\begin{equation}\label{eq:4:3}
\begin{split} 
\partial_{\varrho_1}& \partial_{\varrho_2} \Lambda_3\bigl[\varrho_1, \varrho_2, \Theta^{i}_1[\varrho_1,\varrho_2],\Theta^{i}_2[\varrho_1,\varrho_2], \Theta^{o}[\varrho_1,\varrho_2], \Xi[\varrho_1,\varrho_2]\bigr](x)\\
&= \int_{\partial \Omega^o}S_2(x-y)  \partial_{\varrho_1}\partial_{\varrho_2}\Theta^{o}[\varrho_1,\varrho_2](y)\, d\sigma_y + \partial_{\varrho_1}\partial_{\varrho_2}\Xi[\varrho_1,\varrho_2]\\
&+ \sum_{j=1}^2\int_{\partial \Omega^i_j}S_2(x-\varrho_1p^j-\varrho_1 \varrho_2 s)\Theta^{i}_j[\varrho_1,\varrho_2](s)\, d\sigma_s\\
&-\varrho_1\sum_{h,j=1}^2 \int_{\partial \Omega^i_j}\Big[\big(\partial_h S_2\big)(x-\varrho_1p^j-\varrho_1 \varrho_2 s)\Big](p^j+\varrho_2 s)_h\Theta^{i}_j[\varrho_1,\varrho_2](s)\, d\sigma_s\\ 
&+\varrho_1 \sum_{j=1}^2\int_{\partial \Omega^i_j}S_2(x-\varrho_1p^j-\varrho_1 \varrho_2 s)\partial_{\varrho_1}\Theta^{i}_j[\varrho_1,\varrho_2](s)\, d\sigma_s\\
&-2\varrho_1 \varrho_2\sum_{h,j=1}^2 \int_{\partial \Omega^i_j}\bigg[\Big(\partial_h S_2\Big)(x-\varrho_1p^j-\varrho_1 \varrho_2 s)\bigg]s_h\Theta^{i}_j[\varrho_1,\varrho_2](s)\, d\sigma_s\\
&+\varrho_1^2 \varrho_2\sum_{h,j,k=1}^2 \int_{\partial \Omega^i_j}\bigg[\Big(\partial_h\partial_kS_2\Big)(x-\varrho_1p^j-\varrho_1 \varrho_2 s)\bigg]s_h(p^j+\varrho_2 s)_k\Theta^{i}_j[\varrho_1,\varrho_2](s)\, d\sigma_s\\
&-\varrho_1^2 \varrho_2\sum_{h,j=1}^2 \int_{\partial \Omega^i_j}\bigg[\Big(\partial_h S_2\Big)(x-\varrho_1p^j-\varrho_1 \varrho_2 s)\bigg]s_h\partial_{\varrho_1}\Theta^{i}_j[\varrho_1,\varrho_2](s)\, d\sigma_s\\
&+ \varrho_2\sum_{j=1}^2 \int_{\partial \Omega^i_j}S_2(x-\varrho_1p^j-\varrho_1 \varrho_2 s)\partial_{\varrho_2}\Theta^{i}_j[\varrho_1,\varrho_2](s)\, d\sigma_s\\
&-\varrho_1 \varrho_2\sum_{h,j=1}^2 \int_{\partial \Omega^i_j}\bigg[\Big(\partial_h S_2\Big)(x-\varrho_1p^j-\varrho_1 \varrho_2 s)\bigg](p^j+\varrho_2s)_h \partial_{\varrho_2}\Theta^{i}_j[\varrho_1,\varrho_2](s)\, d\sigma_s\\
&+\varrho_1 \varrho_2\sum_{j=1}^2 \int_{\partial \Omega^i_j}S_2(x-\varrho_1p^j-\varrho_1 \varrho_2 s)\partial_{\varrho_1}\partial_{\varrho_2}\Theta^{i}_j[\varrho_1,\varrho_2](s)\, d\sigma_s  \qquad \forall x \in \partial \Omega^o\, .
\end{split}
\end{equation}
Then by equality \eqref{eq:ansol:der}, by formula \eqref{eq:4:3},  by equality \eqref{intf}, by taking $(\varrho_1,\varrho_2)=(0,0)$, and by classical potential theory (see also Lemma \ref{lem:smp} (ii)), we deduce that $(\theta^{o}_{(1,1)},\xi_{(1,1)})$ is the unique pair in $C^{0,\alpha}(\partial \Omega^o)_0 \times \mathbb{R}$ such that equation \eqref{eq:4:6} holds. \qquad\end{proof}

In Lemmas \ref{lem:7} and \ref{lem:8},  we turn to consider $(\theta^{o}_{(1,2)},\xi_{(1,2)})$ and $(\theta^{o}_{(2,1)},\xi_{(2,1)})$.

\begin{lemma}\label{lem:7}
Let $n=2$. Let $r_\ast=0$. Let the assumptions of Proposition \ref{prop:ansol} hold. Then $\theta^o_{(1,2)}=0$ and $\xi_{(1,2)}=0$.
\end{lemma}
\begin{proof}  If $(\varrho_1,\varrho_2) \in \mathcal{U}$, then by differentiating 
\[
 \partial_{\varrho_1}\partial_{\varrho_2} \Lambda_3\bigl[\varrho_1, \varrho_2, \Theta^{i}_1[\varrho_1,\varrho_2],\Theta^{i}_2[\varrho_1,\varrho_2], \Theta^{o}[\varrho_1,\varrho_2], \Xi[\varrho_1,\varrho_2]\bigr]
 \] 
 for $n=2$ (cf.~equality \eqref{eq:4:3}), we deduce that
\begin{equation}\label{eq:7:1}
\begin{split} 
\partial_{\varrho_1} \partial_{\varrho_2}^2 &\Lambda_3\bigl[\varrho_1, \varrho_2, \Theta^{i}_1[\varrho_1,\varrho_2],\Theta^{i}_2[\varrho_1,\varrho_2], \Theta^{o}[\varrho_1,\varrho_2], \Xi[\varrho_1,\varrho_2]\bigr](x)\\
&= \int_{\partial \Omega^o}S_2(x-y)  \partial_{\varrho_1}\partial_{\varrho_2}^2\Theta^{o}[\varrho_1,\varrho_2](y)\, d\sigma_y + \partial_{\varrho_1}\partial_{\varrho_2}^2\Xi[\varrho_1,\varrho_2]\\
&+2 \sum_{j=1}^2\int_{\partial \Omega^i_j}S_2(x-\varrho_1p^j-\varrho_1 \varrho_2 s)\partial_{\varrho_2}\Theta^{i}_j[\varrho_1,\varrho_2](s)\, d\sigma_s\\
&+\varrho_1 R_{1}[\varrho_1,\varrho_2](x)+\varrho_2 R_{2}[\varrho_1,\varrho_2](x) \qquad \forall x \in \partial \Omega^o\, ,
\end{split}
\end{equation}
where $R_1$, $R_2$ are real analytic maps from $\mathcal{U}$ to $C^{1,\alpha}(\partial \Omega^o)$. Then by equality \eqref{eq:ansol:der}, by formula \eqref{eq:7:1},  by taking $(\varrho_1,\varrho_2)=(0,0)$, and by Lemma \ref{lem:3}, we deduce that $(\theta^{o}_{(1,2)},\xi_{(1,2)})$ is such that 
\[
\int_{\partial \Omega^o}S_2(x-y)\theta^o_{(1,2)}(y)\, d\sigma_y +\xi_{(1,2)}=0 \qquad \forall x \in \partial \Omega^o\, .
\]
Then by Lemma \ref{lem:smp} (ii) we deduce that $(\theta^{o}_{(1,2)},\xi_{(1,2)})=(0,0)$. \qquad\end{proof}

\begin{lemma}\label{lem:8}
Let $n=2$. Let $r_\ast=0$. Let the assumptions of Proposition \ref{prop:ansol} hold. Then $(\theta^{o}_{(2,1)},\xi_{(2,1)})$ is the unique pair in $C^{0,\alpha}(\partial \Omega^o)_0\times \mathbb{R}$ such that
\begin{equation}\label{eq:8:4}
\int_{\partial \Omega^o}S_2(x-y)\theta^o_{(2,1)}(y)\, d\sigma_y +\xi_{(2,1)}=2 \sum_{h,j=1}^2 (p^j)_h\partial_h S_2(x)\int_{\partial \Omega^i_j}f_j\,d\sigma \qquad \forall x \in \partial \Omega^o\, .
\end{equation}
\end{lemma}
\begin{proof}  If $(\varrho_1,\varrho_2) \in \mathcal{U}$, then by differentiating 
\[
\partial_{\varrho_1} \partial_{\varrho_2}\Lambda_3\bigl[\varrho_1, \varrho_2, \Theta^{i}_1[\varrho_1,\varrho_2],\Theta^{i}_2[\varrho_1,\varrho_2], \Theta^{o}[\varrho_1,\varrho_2], \Xi[\varrho_1,\varrho_2]\bigr]
\] 
for $n=2$ (cf.~equality \eqref{eq:4:3}), we deduce that
\begin{equation}\label{eq:8:1}
\begin{split}
\partial_{\varrho_1}^2 &\partial_{\varrho_2} \Lambda_3\bigl[\varrho_1, \varrho_2, \Theta^{i}_1[\varrho_1,\varrho_2],\Theta^{i}_2[\varrho_1,\varrho_2], \Theta^{o}[\varrho_1,\varrho_2], \Xi[\varrho_1,\varrho_2]\bigr](x)\\
&= \int_{\partial \Omega^o}S_2(x-y)  \partial_{\varrho_1}^2\partial_{\varrho_2}\Theta^{o}[\varrho_1,\varrho_2](y)\, d\sigma_y + \partial_{\varrho_1}^2\partial_{\varrho_2}\Xi[\varrho_1,\varrho_2]\\
&- \sum_{h,j=1}^2\int_{\partial \Omega^i_j}\bigg[\Big(\partial_h S_2\Big)(x-\varrho_1p^j-\varrho_1 \varrho_2 s)\bigg](p^j+\varrho_2 s)_h\Theta^{i}_j[\varrho_1,\varrho_2](s)\, d\sigma_s\\
&+ \sum_{j=1}^2\int_{\partial \Omega^i_j}S_2(x-\varrho_1p^j-\varrho_1 \varrho_2 s)\partial_{\varrho_1}\Theta^{i}_j[\varrho_1,\varrho_2](s)\, d\sigma_s\\
&- \sum_{h,j=1}^2\int_{\partial \Omega^i_j}\bigg[\Big(\partial_h S_2\Big)(x-\varrho_1p^j-\varrho_1 \varrho_2 s)\bigg](p^j+\varrho_2 s)_h\Theta^{i}_j[\varrho_1,\varrho_2](s)\, d\sigma_s\\
&+ \sum_{j=1}^2\int_{\partial \Omega^i_j}S_2(x-\varrho_1p^j-\varrho_1 \varrho_2 s)\partial_{\varrho_1}\Theta^{i}_j[\varrho_1,\varrho_2](s)\, d\sigma_s\\&+\varrho_1 R_{3}[\varrho_1,\varrho_2](x)+\varrho_2 R_{4}[\varrho_1,\varrho_2](x) \qquad \forall x \in \partial \Omega^o\, ,
\end{split}
\end{equation}
where $R_3, R_{4}$ are real analytic maps from $\mathcal{U}$ to $C^{1,\alpha}(\partial \Omega^o)$. Then by equality \eqref{eq:ansol:der}, by formula \eqref{eq:8:1},  by taking $(\varrho_1,\varrho_2)=(0,0)$, by Lemma \ref{lem:2}, and by classical potential theory (see also Lemma \ref{lem:smp} (ii)), we deduce that $(\theta^{o}_{(2,1)},\xi_{(2,1)})$ is the unique pair in $C^{0,\alpha}(\partial \Omega^o)_0\times \mathbb{R}$ such that  equation \eqref{eq:8:4} holds. \qquad\end{proof}
 
 We now exploit the previous results to compute an expansion of the sum of the last two terms in the representation formula \eqref{eq:finteq1}. Indeed, by standard calculus, we deduce the validity of the following.

\begin{lemma}\label{lem:10}
Let $n=2$. Let $r_\ast=0$. Let the assumptions of Proposition \ref{prop:ansol} hold. If $x \in \mathrm{cl}\Omega^o$ is fixed, then
\[
\begin{split}
\int_{\partial \Omega^o}S_2(x-y)&\Theta^o[\varrho_1,\varrho_2](y)\, d\sigma_y +\Xi[\varrho_1,\varrho_2]=
u_{(0,0)}(x)+\varrho_1\varrho_2 u_{(1,1)}(x)\\
&+\frac{1}{2}\varrho_1^2\varrho_2u_{(2,1)}(x)+O(|\varrho_1^3\varrho_2|+|\varrho_1^2\varrho_2^2|+|\varrho_1\varrho_2^3|)\, ,
\end{split}
\]
as $(\varrho_1,\varrho_2)$ tends to $(0,0)$, where
\[
u_{(j,k)} \equiv v^+[\partial \Omega^o,\theta^o_{(j,k)}] +\xi_{(j,k)} \qquad  \forall (j,k) \in \{(0,0), (1,1), (2,1)\}\, .
\]
\end{lemma}

Instead, in the following lemma, we consider the remaining part of formula \eqref{eq:finteq1}.
 
 \begin{lemma}\label{lem:11}
Let $n=2$. Let $r_\ast=0$. Let the assumptions of Proposition \ref{prop:ansol} hold. Let $x \in \mathrm{cl}\Omega^o \setminus \{0\}$ be fixed. Then we have
 \begin{equation}\label{eq:lem11:0}
\begin{split}
&\sum_{j=1}^2 \int_{\partial \Omega^i_j}S_2(x-\varrho_1 p^j -\varrho_1 \varrho_2 s) \Theta^i_j[\varrho_1,\varrho_2](s)\, d\sigma_s\\
&=S_2(x)\sum_{j=1}^2 \int_{\partial \Omega^i_j}f_j\, d\sigma -\varrho_1\sum_{h,j=1}^2 (\partial_h S_2)(x)(p^j)_h \int_{\partial \Omega^i_j}f_j\, d\sigma+O(|\varrho_1^2|+|\varrho_1\varrho_2|+|\varrho_2^2|) \, ,
\end{split}
 \end{equation}
as $(\varrho_1,\varrho_2)$ tends to $(0,0)$. 
\end{lemma}
\begin{proof}  By arguing as in the proof of Theorem \ref{thm:rep}, one verifies that the left hand side of equality \eqref{eq:lem11:0} defines a real analytic function in the variable $(\varrho_1,\varrho_2)$ in a sufficiently small neighborhood of $(0,0)$. We have
 \begin{equation}\label{eq:lem11:1}
 \begin{split}
 \partial_{\varrho_1}&\sum_{j=1}^2 \int_{\partial \Omega^i_j}S_2(x-\varrho_1 p^j -\varrho_1 \varrho_2 s) \Theta^i_j[\varrho_1,\varrho_2](s)\, d\sigma_s \\
 &=-\sum_{h,j=1}^2 \int_{\partial \Omega^i_j}\bigg[\Big(\partial_h S_2\Big)(x-\varrho_1p^j-\varrho_1 \varrho_2 s)\bigg](p^j+\varrho_2 s)_h \Theta^i_j[\varrho_1,\varrho_2](s)\, d\sigma_s\\
 &+\sum_{j=1}^2\int_{\partial \Omega^i_j}S_2(x-\varrho_1p^j-\varrho_1\varrho_2 s)\partial_{\varrho_1}\Theta^i_j[\varrho_1,\varrho_2](s)\, d\sigma_s\, .
 \end{split}
 \end{equation}
 Then for $(\varrho_1,\varrho_2)=(0,0)$ the right hand side of equality \eqref{eq:lem11:1} becomes
 \[
 -\sum_{h,j=1}^2 (\partial_h S_2)(x)(p^j)_h \int_{\partial \Omega^i_j}f_j\, d\sigma
 \]
(cf.~equality \eqref{intf} and Lemma \ref{lem:2}). Similarly,
 \begin{equation}\label{eq:lem11:2}
 \begin{split}
 \partial_{\varrho_2}&\sum_{j=1}^2 \int_{\partial \Omega^i_j}S_2(x-\varrho_1 p^j -\varrho_1 \varrho_2 s) \Theta^i_j[\varrho_1,\varrho_2](s)\, d\sigma_s \\
 &=-\varrho_1\sum_{h,j=1}^2 \int_{\partial \Omega^i_j}\bigg[\Big(\partial_h S_2\Big)(x-\varrho_1p^j-\varrho_1 \varrho_2 s)\bigg]s_h \Theta^i_j[\varrho_1,\varrho_2](s)\, d\sigma_s\\
 &+\sum_{j=1}^2\int_{\partial \Omega^i_j}S_2(x-\varrho_1p^j-\varrho_1\varrho_2 s)\partial_{\varrho_2}\Theta^i_j[\varrho_1,\varrho_2](s)\, d\sigma_s\, ,
 \end{split}
\end{equation} 
and the right hand side of \eqref{eq:lem11:2} equals $0$ for $(\varrho_1,\varrho_2)=(0,0)$ (cf.~Lemma \ref{lem:3}).  As a consequence,  by standard calculus, we deduce the validity of the lemma.\qquad\end{proof} 
 
 Finally, by combining Lemmas \ref{lem:10} and \ref{lem:11}, we deduce the validity of the main result of this section.
 
 \begin{proposition}
 Let $n=2$. Let $r_\ast=0$. Let the assumptions of Proposition \ref{prop:ansol} hold. Let $u_{(j,k)}$ be as in Lemma \ref{lem:10} for all  $(j,k) \in \{(0,0), (1,1), (2,1)\}$. Let $x \in \mathrm{cl}\Omega^o \setminus \{0\}$ be fixed. Then we have
  \[
 \begin{split}
 u[\varrho_1,\varrho_2](x)&=u_{(0,0)}(x)+\varrho_1 \varrho_2 \Big(u_{(1,1)}(x)+S_2(x)\sum_{j=1}^2 \int_{\partial \Omega^i_j}f^j\, d\sigma\Big)\\
 &+\varrho_1^2\varrho_2\Big(\frac{1}{2}u_{(2,1)}(x)-\sum_{h,j=1}^2\partial_h S_2(x)(p^j)_h \int_{\partial \Omega^i_j}f^j\, d\sigma\Big)\\
 &+O(|\varrho_1^3 \varrho_2|+|\varrho_1^2 \varrho_2^2|+|\varrho_1 \varrho_2^3|)\, ,
 \end{split}
 \]
as $(\varrho_1,\varrho_2)$ tends to $(0,0)$. 
 \end{proposition}
 
 \begin{remark}
 If we further assume that $\int_{\partial \Omega^i_j}f_j \, d\sigma=0$ for all $j \in \{1,2\}$ then we can deduce the existence of functions $\tilde{u}_{(3,1)}$, $\tilde{u}_{(2,2)}$,  and $\tilde{u}_{(1,3)}$ such that
 \[
 \begin{split}
  u[\varrho_1,\varrho_2](x)&=u_{(0,0)}(x)+O(|\varrho_1^3 \varrho_2|+|\varrho_1^2 \varrho_2^2|+|\varrho_1 \varrho_2^3|)\\
&=u_{(0,0)}(x)+\varrho_1^3 \varrho_2 \tilde{u}_{(3,1)}(x)+\varrho_1^2 \varrho_2^2 \tilde{u}_{(2,2)}(x)\\
&+\varrho_1 \varrho_2^3 \tilde{u}_{(1,3)}(x)+ O(|\varrho_1^4 \varrho_2|+|\varrho_1^3 \varrho_2^2|+|\varrho_1^2 \varrho_2^3|+|\varrho_1 \varrho_2^4|)  \, ,
\end{split}
 \]
as $(\varrho_1,\varrho_2)$ tends to $(0,0)$. 
 \end{remark}
 
 \section*{Acknowledgment}
The authors wish to thank V.~Bonnaillie-No\"el, M.~Dambrine, and C.~Lacave for several useful discussions. The work of M.~Dalla Riva and P.~Musolino is supported by ``Progetto di Ateneo: Singular perturbation problems for differential operators -- CPDA120171/12" of the University of Padova. The research  of M.~Dalla Riva  was supported by Portuguese funds through the CIDMA - Center for Research and Development in Mathematics and Applications, and the Portuguese Foundation for Science and Technology (``FCT--Funda{\c c}{\~a}o para a Ci\^encia e a Tecnologia''), within project UID/MAT/04106/2013. M.~Dalla Riva acknowledges also the support from HORIZON 2020 MSC EF project FAANon (grant agreement MSCA-IF-2014-EF - 654795) at the University of Aberystwyth, UK. P.~Musolino is member of the Gruppo Nazionale per l'Analisi Matematica, la Probabilit\`a e le loro Applicazioni (GNAMPA) of the Istituto Nazionale di Alta Matematica (INdAM) and acknowledges the support of  ``INdAM GNAMPA Project 2015 - Un approccio funzionale analitico per problemi di perturbazione singolare e di omogeneizzazione".


\begin{thebibliography}{11}

\bibitem{AmKa07}
H.~Ammari and H.~Kang, {\em Polarization and moment tensors}, volume 162 of {\em Applied
  Mathematical Sciences}, Springer, New York, 2007.

\bibitem{BoDa13} V.~Bonnaillie-No\"el and M.~Dambrine,
{\em Interactions between moderately close circular inclusions: the Dirichlet-Laplace equation in the plane},  Asymptot. Anal., {\bf 84} (2013), 197--227.

\bibitem{BoDaLa} V.~Bonnaillie-No\"el, M.~Dambrine, and C.~Lacave, {\em Interactions Between Moderately Close Inclusions for the Two-Dimensional Dirichlet--Laplacian}, Appl. Math. Res. Express. AMRX, to appear. DOI: 10.1093/amrx/abv008


\bibitem{BoDaToVi07} V.~Bonnaillie-No\"el, M.~Dambrine, S.~Tordeux, and G.~Vial, {\em On moderately close inclusions for the Laplace equation.}, C.~R.~Math.~Acad.~Sci.~Paris, \textbf{19}  (2007),  609--614. 

\bibitem{BoDaToVi09} V.~Bonnaillie-No\"el, M.~Dambrine, S.~Tordeux, and G.~Vial, {\em Interactions between moderately close inclusions for the Laplace equation}, Math.~Models Methods Appl.~Sci., \textbf{19}  (2009),  1853--1882. 


 \bibitem{BoLaMa} V.~Bonnaillie-No\"el, C.~Lacave, and N.~Masmoudi, {\em Permeability through a perforated domain for the incompressible 2D Euler equations},  Ann. Inst. H. Poincar\'e Anal. Non Lin\'eaire,  \textbf{32}  (2015),  159--182.
 
 
 \bibitem{BuPr} D.~Buoso and L.~Provenzano, {\em A few shape optimization results for a biharmonic Steklov problem}, J.~Differential Equations, \textbf{259} (2015), 1778--1818.


\bibitem{ChCl14} L.~Chesnel and X.~Claeys {\em A numerical approach for the Poisson equation
in a planar domain with a small inclusion}, submitted. arXiv:1410.3508

  \bibitem{Da13}
M.~Dalla~Riva, {\em Stokes flow in a singularly perturbed exterior domain},  Complex Var. Elliptic Equ., {\bf 58} (2013), 231--257.

\bibitem{DaLa10a}
M.~Dalla~Riva and  M.~Lanza~de~Cristoforis,   {\em  Microscopically weakly singularly
perturbed loads for a nonlinear traction boundary value problem.  A functional 
analytic approach},  Complex Var.~Elliptic Equ., {\bf 55} (2010), 771--794.

\bibitem{DaMu12}
M.~Dalla~Riva and  P.~Musolino,   {\em  Real analytic families of harmonic functions in a domain 
with a small hole}, J.~Differential Equations,  {\bf 252} (2012), 6337--6355.

\bibitem{DaMu15}
M.~Dalla~Riva and  P.~Musolino,   {\em  Real analytic families of harmonic functions in a planar domain with a small hole}, J.~Math.~Anal~ Appl., {\bf 422} (2015), 37--55.

\bibitem{DaMuRo}
M.~Dalla~Riva,  P.~Musolino, and S.V.~Rogosin,   {\em  Series expansions for the solution of the Dirichlet problem in a planar domain with a small hole}, Asymptot. Anal.,  {\bf 92} (2015), 339--361.

\bibitem{DaToVi10}
M.~Dauge, S.~Tordeux, and  G.~Vial,  {\em Selfsimilar perturbation near a corner: matching versus multiscale expansions for a model problem}. In { \em Around the research of Vladimir Maz'ya. II}, 95--134, Int. Math. Ser. (N. Y.), 12, Springer, New York, 2010.

\bibitem{De85}
K.~Deimling, {\em Nonlinear functional analysis}, Springer-Verlag, Berlin, 1985.



\bibitem{Fo95} 
G.B.~Folland, {\em Introduction to partial differential equations}, Second edition,
  Princeton University Press, Princeton N.J., 1995.




\bibitem{GiTr83}
D.~Gilbarg and N.S.~Trudinger, {\em Elliptic partial 
differential equations of second order},   Springer Verlag, Berlin, \textit{etc.}, 1983.


\bibitem{Il78} 
A.M.~Il'in, {\em A boundary value problem for a second-order elliptic equation
in a domain with a narrow slit. I. The two-dimensional case}, Math.~USSR Sb., \textbf{28} (1978),  pp.~459--480. 

\bibitem{Il92}
A.M.~Il'in,  {\em Matching of asymptotic expansions of solutions of boundary value problems}, Translations of Mathematical Monographs 102, American Mathematical Society, Providence, 1992.

\bibitem{Il99}
A.M.~Il'in, {\em The boundary layer}, in: Fedoryuk MV (ed.) {\em Partial Differential Equations. V. Asymptotic Methods for Partial 
Differential Equations}, {Encylopaedia of Mathematical Sciences} \textbf{34}, Springer-Verlag, Berlin, 1999, pp.~173--210.


\bibitem{KoMaMo99}
V.~Kozlov, V.~Maz'ya, and A.~Movchan, 
  \textit{Asymptotic analysis of fields in multi-structures},
 Oxford Mathematical Monographs, the Clarendon Press Oxford University
  Press, New York, 1999. 
  
 \bibitem{LaLa04} P.D.~Lamberti and M.~Lanza de Cristoforis, {\em A real analyticity result for symmetric functions of the eigenvalues of a domain dependent Dirichlet problem for the Laplace operator}, J. Nonlinear Convex Anal., {\bf 5} (2004), 19--42.

\bibitem{La07}
M.~Lanza~de~Cristoforis,  {\em Asymptotic behavior of the solutions of a nonlinear Robin problem
for the Laplace operator in a domain with a small hole: a functional 
analytic approach}, Complex Var.~Elliptic Equ., {\bf 52} (2007), 945--977. 

\bibitem{La10}
M.~Lanza~de Cristoforis, {\em Asymptotic behavior of the solutions of a non-linear transmission
  problem for the {L}aplace operator in a domain with a small hole. {A}
  functional analytic approach}, Complex Var. Elliptic Equ., {\bf 55} (2010), 269--303.

\bibitem{LaMu13}
M.~Lanza~de~Cristoforis and P.~Musolino, {\em 
A real analyticity result for a nonlinear integral operator},  J. Integral Equations Appl., {\bf 25} (2013), 21--46.


\bibitem{LaRo04}
M.~Lanza~de~Cristoforis and L.~Rossi, {\em 
Real analytic dependence of simple and double 
layer potentials upon perturbation 
of the support and of the density}, J. Integral Equations 
Appl.,  {\bf 16} (2004), 137--174.



\bibitem{MaMoNi13}
V.~Maz'ya, A.~Movchan, and M.~Nieves,  {\em Green's kernels and meso-scale approximations in perforated domains},  Lecture Notes in Mathematics {\bf 2077},  Springer, Berlin, 2013. 

 \bibitem{MaNaPl00}
V.~Maz'ya, S.~Nazarov, and B.~Plamenevskij, {\em Asymptotic theory of elliptic boundary value problems in
  singularly perturbed domains. {V}ols. {I}, II}, volumes 111, 112 of {\em Operator
  Theory: Advances and Applications}, Birkh{\"a}user Verlag, Basel, 2000.


\bibitem{Mi65}
C.~Miranda, {\em Sulle propriet\`{a} di regolarit\`{a} di certe 
trasformazioni integrali}, Atti Accad.~Naz.~Lincei Mem.~Cl.~Sci.~Fis.~Mat.~Natur.~Sez. I, {\bf 7} (1965), 303--336. 


\bibitem{NoSo13} A.A.~Novotny and J.~Soko\l owski, {\em Topological derivatives in shape optimization}, Interaction of Mechanics and Mathematics, Springer, Heidelberg, 2013.


\end{thebibliography}
\end{document}